\documentclass[10pt,a4paper,twoside,english]{scrartcl}
\usepackage[T1]{fontenc}
\usepackage[latin1]{inputenc}
\usepackage{xcolor}
\usepackage{pdfcolmk}
\usepackage{babel}
\usepackage{verbatim}
\usepackage{url}
\usepackage{amsthm}
\usepackage{amsmath}
\usepackage{amssymb}
\PassOptionsToPackage{normalem}{ulem}
\usepackage{ulem}
\usepackage[unicode=true,pdfusetitle,
 bookmarks=true,bookmarksnumbered=false,bookmarksopen=false,
 breaklinks=false,pdfborder={0 0 1},backref=false,colorlinks=false]
 {hyperref}

\makeatletter

\pdfpageheight\paperheight
\pdfpagewidth\paperwidth

\providecolor{lyxadded}{rgb}{0,0,1}
\providecolor{lyxdeleted}{rgb}{1,0,0}

\theoremstyle{plain}
\newtheorem{thm}{\protect\theoremname}[section]
  \theoremstyle{remark}
  \newtheorem{rem}[thm]{\protect\remarkname}
  \theoremstyle{definition}
  \newtheorem{example}[thm]{\protect\examplename}
  \theoremstyle{definition}
  \newtheorem{defn}[thm]{\protect\definitionname}
  \theoremstyle{plain}
  \newtheorem{cor}[thm]{\protect\corollaryname}
  \theoremstyle{plain}
  \newtheorem{lem}[thm]{\protect\lemmaname}

\usepackage[T1]{fontenc}    

\usepackage{a4wide}        
\usepackage{amsmath}
\usepackage{enumerate}

\usepackage{euscript}
\usepackage{mathtools}
\allowdisplaybreaks[1] 
\usepackage{amsfonts}
\newenvironment{keywords}{ \noindent\footnotesize\textbf{Keywords and phrases:}}{}

\newenvironment{class}{\noindent\footnotesize\textbf{Mathematics subject classification 2000:}}{}

\usepackage{color,tu-preprint}

\newcommand*{\trace}{\operatorname{trace}}

\newcommand*{\dive}{\operatorname{div}}

\newcommand*{\Grad}{\operatorname{Grad}}
\newcommand*{\Div}{\operatorname{Div}}

\renewcommand*{\i}{\mathrm{i}}

\DeclareMathAccent{\Circ}{\mathalpha}{operators}{"17}
\newcommand{\interior}[1]{\Circ{#1}}
\DeclareMathAccent{\fiss}{\mathalpha}{operators}{"15}

\renewcommand{\Im}{\operatorname{\mathfrak{Im}}}
\renewcommand{\Re}{\operatorname{\mathfrak{Re}}}
\newcommand{\sym}{\operatorname{sym}}
\renewcommand{\skew}{\operatorname{skew}}

\newcommand{\lspan}{\operatorname{span}}

\newcommand{\oi}[2]{\left]#1,#2 \right[}
\newcommand{\rga}[1]{\left]#1,\infty  \right[}

\newcommand{\lci}[2]{\left[#1,#2 \right[}

\renewcommand{\tilde}{\widetilde}
\renewcommand*{\epsilon}{\varepsilon}
\renewcommand*{\theta}{\vartheta}

\makeatother

  \providecommand{\corollaryname}{Corollary}
  \providecommand{\definitionname}{Definition}
  \providecommand{\examplename}{Example}
  \providecommand{\lemmaname}{Lemma}
  \providecommand{\remarkname}{Remark}
\providecommand{\theoremname}{Theorem}

\begin{document}
\institut{Institut f\"ur Analysis}

\preprintnumber{MATH-AN-12-2013}

\preprinttitle{On Some Models for Elastic Solids with Micro-Structure.}

\author{Rainer Picard, Sascha Trostorff \& Marcus Waurick} 

\makepreprinttitlepage

\setcounter{section}{-1}

\date{}

\title{On Some Models for Elastic Solids with Micro-Structure.}

\author{Rainer Picard, Sascha Trostorff \& Marcus Waurick \thanks{Institut für Analysis,Fachrichtung Mathematik, Technische Universität Dresden, Germany, rainer.picard@tu-dresden.de, sascha.trostorff@tu-dresden.de, marcus.waurick@tu-dresden.de} }
\maketitle
\begin{abstract}
\noindent \textbf{Abstract.} We review the concept of well-posedness
in the context of evolutionary problems from mathematical physics
for a particular subclass of problems from elasticity theory. The
complexity of physical phenomena appears as encoded in so called material
laws. The usefulness of the structural perspective developed is illustrated
by showing that many initial boundary value problems in the theory
of elastic solids share the same type of solution theory. Moreover,
interconnections of the respective models are discussed via a previously
introduced mother/descendant mechanism.
\end{abstract}

\begin{class}
35F10 Initial value problems for linear first-order PDE, linear evolution equations,
47B25 Symmetric and selfadjoint operators (unbounded),
47F05 Partial differential operators,
47G20 Integro-differential operators,
35L90 Abstract hyperbolic evolution equations,
74B05 Classical linear elasticity,
74H20 Existence of solutions,
74H25 Uniqueness of solutions,
74M25 Micromechanics.
\end{class}\setcounter{page}{1}
\newpage

\tableofcontents{}

\newpage
\setcounter{section}{-1}

\section{Introduction}

In more complex considerations of elasticity theory of compound media
their internal fine structure needs to be accounted for. The fundamental
concept originated with the work of the Cosserat brothers, \cite{40.0862.02}.
At the time, however, their ideas did not receive the immediate widespread
attention of the mechanics community. The idea to account for additional
degrees of freedom stemming from the fine structure of media was to
have an additional set of elastic equations on the micro-level coupling
with an elastic equation on the macro-level. Most significantly the
particular coupling induces non-symmetry into the strain tensor, which
was a standard assumption up to then. In the 60's of the last century
such complex media became the focus of attention resulting in rediscovery,
revival and further development of the ideas of the Cosserat brothers.
We mention here in particular the seminal contributions by Aero and
Kuvshinski \cite{MR0162404}, Mindlin \cite{0119.40302}, Nowacki
\cite{zbMATH03456512,zbMATH03490546} and Eringen \cite{zbMATH03352829}.
For a more detailed historical and bibliographic survey see \cite[Chapter I]{zbMATH06233830}. 

The intention of this paper is to investigate some of the models concerned
with fine-structural elasticity effects in the framework of a theory
which allows for a unified perspective. Indeed, it has been found
that many classical equations of mathematical physics share a common
form, see e.g. \cite{pre05919306}. The general form can be described
loosely as follows: we look for $U$ and $V$ satisfying 
\begin{equation}
\dot{V}+AU=f,\label{eq:dyn}
\end{equation}
where $\dot{V}$ stands for the time-derivative of $V$, $f$ is a
given forcing term and $A$ is a (usually unbounded) operator, which
in standard situations is skew-selfadjoint in a suitable Hilbert space
setting. Of course the latter equation is under-determined and a second
equation -- the so-called material law or constitutive relation --
linking $U$ and $V$ via a bounded linear operator $\mathcal{M}$
acting in space-time has to be supplied:
\begin{equation}
V=\mathcal{M}U.\label{mat-law}
\end{equation}
Substituting the material law into the first equation, we arrive at
\[
\left(\mathcal{M}U\right)^{\cdot}+AU=f.
\]
In \cite{1200.35050}, a well-posedness result for equations of this
form has been shown in a space-time Hilbert space setting. This result
has been generalized in various directions, \cite{Kalauch2011,Picard2012,PicTroWau2012c,RainerPicard2013,S.Trostorff,Trostorff2013,Waurick2013c}.
In these references many particular examples for the problem class
under consideration have been given. The approach of proving well-posedness
consists in establishing the time-derivative as a continuously invertible
normal operator and proving strict positive definiteness of both the
operator sum and its adjoint in a space-time Hilbert space setting.
On the one hand the concept of a solution is rather weak for a solution
only belongs to the domain of the closure of the respective operator
sum (akin of what is called a mild solution). On the other hand this
concept does not rely on the existence of a fundamental solution as
utilized in the semi-group approach, allowing for more general material
laws. 

Temporal processes are distinguish from purely spatial phenomena by
causality. In the Hilbert space context causality can conveniently
be analyzed with the help of the Paley-Wiener theorem (see e.g. \cite[Theorem 19.2]{Rudin})
employing the (vector-valued) Plancherel formula, which is only valid
if the functions under consideration attain values in a Hilbert space,
\cite{Kwapien}. 

In this article we discuss several models of elasticity (based on
considerations in \cite{0119.40302,0164.27507,zbMATH03490546,0604.73020,1104.74007,1136.74018})
and prove well-posedness of corresponding initial boundary value problems
by assuming positive definiteness conditions on the operators being
contained in the material law. The key is to show that the models
discussed fit into the general scheme of evolutionary equations. Moreover,
rather than proposing an energy from which the equations are derived,
based on the above structural observation we approach the equation
directly, the ``energy'' being essentially just the underlying norm.
We show, indeed, a particular conservation property for an abstract
problem class, suited to cover the models of elastic media considered.
Due to the power of the general machinery we are not merely reproducing
known results obtained with different approaches but generalizing
these models to general inhomogeneous anisotropic media, including
spatially non-local media, as considered in \cite{zbMATH01817639}.

Some of the models are linked by a formal simplification. This idea
has found its rigorous justification in terms of a so-called ``mother''
and ``descendant'' mechanism, \cite{Picard201354}. We will show
that several models for elastic solids can be regarded as linked by
this mechanism, which  may yield further insight in the interconnections
between these models.

In Section \ref{sec:Setting} we briefly recall the functional analytic
setting in which we want to establish a solution theory for our problem
class. In Section \ref{sec:Well-po_and_Mother} the abstract problem
class is described and the solution theory for this class is recalled.
In Subsection \ref{sub:WP:Linear-Evolutionary-Equations} we also
prove an energy balance equality by anticipating strategies developed
in \cite{PicTroWau2012,PicTroWau2012d}. We recall the results from
\cite{Picard201354} in Subsection \ref{sub:A-Mechanism-Deriving}
to explain the mechanism of deriving (well-posed) evolutionary equations
from given ones. In particular we define what it means for a model
to be a ``descendant'', ``relative'' or ``mother'' of another
one. Section \ref{sec:Cosserat-Elasticity} contains the description
and the well-posedness results for several equations modeling deformable
solids. Proving a well-posedness result for the equations for micromorphic
media in Subsection \ref{sub:Micromorphic-Media}, we show that the
equations of Cosserat elasticity (micropolar media) are a descendant
of the equations for micromorphic media in Subsection \ref{sub:Micropolar-or-Cosserat}.
Other descendants of the system for micromorphic media are discussed
in Subsection \ref{sub:Other-Descendants-of}. Before giving some
concluding remarks in Section \ref{sec:Conclusion}, we also describe
the so-called microstretch model from \cite{0718.73014,0164.27507},
which on the one hand fits into the general problem class but is on
the other hand not a descendant of the model for micromorphic media.

\section{A Hilbert Space Setting for a Class Evolutionary Problems\label{sec:Setting}}

The family of Hilbert spaces $\left(H_{\rho,0}\left(\mathbb{R},H\right)\right)_{\rho\in\mathbb{R}}$,
$H$ Hilbert space, with $H_{\rho,0}\left(\mathbb{R},H\right)\coloneqq L^{2}\left(\mathbb{R},\mu_{\rho},H\right)$,
where the measure $\mu_{\rho}$ is defined by $\mu_{\rho}\left(S\right)\coloneqq\int_{S}\exp\left(-2\rho t\right)\: dt$,
$S\subseteq\mathbb{R}$ a Borel set, $\rho\in\mathbb{R}$, provides
the desired Hilbert space setting for evolutionary problems (cf. \cite{pre05919306,Kalauch2011}).
The sign of $\rho$ is associated with the direction of causality,
where the positive sign is linked to forward causality. Since we have
a preference for forward causality, we shall usually assume that $\rho\in\rga0$.
By construction of these spaces we can establish
\begin{align*}
\exp\left(-\rho\mathbf{m}_{0}\right):H_{\rho,0}\left(\mathbb{R},H\right) & \to H_{0,0}\left(\mathbb{R},H\right)(=L^{2}\left(\mathbb{R},H\right))\\
\varphi & \mapsto\exp\left(-\rho\mathbf{m}_{0}\right)\varphi
\end{align*}
where $\left(\exp\left(-\rho\mathbf{m}_{0}\right)\varphi\right)\left(t\right)\coloneqq\exp\left(-\rho t\right)\varphi\left(t\right)$,
$t\in\mathbb{R}$, as a unitary mapping. We use $\mathbf{m}_{0}$
as a notation for the multiplication-by-argument operator corresponding
to the time parameter.

In this Hilbert space setting the time-derivative operation $\varphi\mapsto\dot{\varphi}$
generates a normal operator $\partial_{0,\rho}$ with%
\footnote{Recall that for normal operators $N$ in a Hilbert space $H$
\begin{align*}
\Re N & \coloneqq\frac{1}{2}\overline{\left(N+N^{*}\right)},\\
\Im N & \coloneqq\frac{1}{2\i}\overline{\left(N-N^{*}\right)}
\end{align*}
and
\[
N=\Re N+\i\Im N.
\]
It is
\[
D\left(N\right)=D\left(\Re N\right)\cap D\left(\Im N\right).
\]
}
\begin{align*}
\Re\partial_{0,\rho} & =\rho,\\
\Im\partial_{0,\rho} & =\frac{1}{\i}\left(\partial_{0,\rho}-\rho\right).
\end{align*}
The selfadjoint operator $\Im\partial_{0,\rho}$ is unitarily equivalent
to the differentiation operator $\frac{1}{\i}\partial_{0,0}$ in $L^{2}\left(\mathbb{R},H\right)=H_{0}\left(\mathbb{R},H\right)$
with domain $H^{1}(\mathbb{R},H)$ - the space of weakly differentiable
functions in $L^{2}(\mathbb{R},H)$ - via
\[
\Im\partial_{0,\rho}=\exp\left(\rho\mathbf{m}_{0}\right)\frac{1}{\i}\partial_{0,0}\exp\left(-\rho\mathbf{m}_{0}\right)
\]
and has the Fourier-Laplace transformation as spectral representation,
which is the unitary transformation
\[
\mathcal{L}_{\rho}\coloneqq\mathcal{F}\:\exp\left(-\rho\mathbf{m}_{0}\right):H_{\rho,0}(\mathbb{R},H)\to L^{2}(\mathbb{R},H),
\]
where $\mathcal{F}:L^{2}(\mathbb{R},H)\to L^{2}(\mathbb{R},H)$ is
the Fourier transformation. Indeed, this follows from the well-known
fact that $\mathcal{F}$ is unitary in $L^{2}\left(\mathbb{R},H\right)$
and a spectral representation for $-\i$ times the derivative in $L^{2}\left(\mathbb{R},H\right)$.
Recall that for continuous functions $\varphi$ with compact support
in $\mathbb{R}$ we have
\[
\left(\mathcal{F}\varphi\right)\left(s\right)\coloneqq\frac{1}{\sqrt{2\pi}}\int_{\mathbb{R}}\exp\left(-\i st\right)\:\varphi\left(t\right)\: dt\quad(s\in\mathbb{R}).
\]
In particular, we have 
\begin{align*}
\Im\partial_{0,\rho} & =\mathcal{L}_{\rho}^{*}\mathbf{m}_{0}\mathcal{L}_{\rho}
\end{align*}
and thus 
\[
\partial_{0,\rho}=\mathcal{L}_{\rho}^{*}\left(\i\mathbf{m}_{0}+\rho\right)\mathcal{L}_{\rho}.
\]

It is crucial to note that for $\rho\not=0$ we have that $\partial_{0,\rho}$
has a bounded inverse. If $\rho>0$ we find from $\Re\partial_{0,\rho}=\rho$
that
\begin{equation}
\left\Vert \partial_{0,\rho}^{-1}\right\Vert _{\rho,0}=\left|\partial_{0,\rho}^{-1}\right|_{\rho,\mathrm{Lip}}=\frac{1}{\rho},\label{eq:time-integral-norm}
\end{equation}
where $\left|\;\cdot\:\right|_{\rho,\mathrm{Lip}}$ denotes the semi-norm
given by associating the smallest Lipschitz constant with a Lipschitz
continuous mapping in $H_{\rho,0}\left(\mathbb{R},H\right)$, which
is for linear operators actually a norm coinciding with the operator
norm denoted by $\left\Vert \:\cdot\;\right\Vert _{\rho,0}$, $\rho\in\rga0$.
For continuous functions $\varphi$ with compact support we find
\begin{equation}
\left(\partial_{0,\rho}^{-1}\varphi\right)\left(t\right)=\int_{-\infty}^{t}\varphi\left(s\right)\: ds,\: t\in\mathbb{R},\rho\in\oi{0}{\infty},\label{eq:integral}
\end{equation}

which shows the causality of $\partial_{0,\rho}^{-1}$ for $\rho>0$.%
\footnote{If $\rho<0$ the operator $\partial_{0,\rho}$ is also boundedly invertible
and its inverse is given by 
\[
\left(\partial_{0,\rho}^{-1}\varphi\right)(t)=-\intop_{t}^{\infty}\varphi(s)\, ds\quad(t\in\mathbb{R})
\]
for all $\varphi\in\interior C_{\infty}(\mathbb{R},H)$, the space
of indefinitely differentiable, compactly supported $H$-valued functions.
Thus, $\rho<0$ corresponds to the backward causal (or anticausal)
case.%
} Since it is usually clear from the context which $\rho$ has been
chosen, we shall, as is customary, drop the index $\rho$ from the
notation for the time derivative and simply use $\partial_{0}$ instead
of $\partial_{0,\rho}$.

\section{A Well-Posed Problem Class of Autonomous Evolutionary Equations\label{sec:Well-po_and_Mother}}

\subsection{Linear Evolutionary Equations Describing Lossless Wave Propagation
Phenomena\label{sub:WP:Linear-Evolutionary-Equations}}

We recall from \cite{1200.35050} (and the concluding chapter of \cite{pre05919306})
that the common form of standard initial boundary value problems of
mathematical physics is given by
\begin{equation}
\overline{\left(\partial_{0}M\left(\partial_{0}^{-1}\right)+A\right)}U=f,\label{eq:evo2}
\end{equation}
where for the rest of the paper we restrict our attention to the simple
case when $A$ is the canonical skew-selfadjoint extension to $H_{\rho,0}\left(\mathbb{R},H\right)$
of an skew-selfadjoint operator%
\footnote{Since this appears to be a matter of debate, we emphasize here that,
as one commonly considers ``adjoint'' as a binary concept and self-adjoint
as a unary notion, we shall treat the ``skew'' situation analogously.
So, we shall say that $A$ is skew-adjoint to $B$ (or is the skew-adjoint
of $B$) if
\[
A=-B^{*}
\]
and we shall call $A$ skew-selfadjoint if it is its own skew-adjoint,
i.e. 
\[
A=-A^{*}.
\]
} in $H$ and the so-called material law operator $M\left(\partial_{0}^{-1}\right)$
is assumed to be of the simple polynomial form
\begin{equation}
M\left(\partial_{0}^{-1}\right)\coloneqq M_{0}+\partial_{0}^{-1}M_{1}+\partial_{0}^{-2}M_{2}\label{eq:material}
\end{equation}
with $M_{0},M_{1},M_{2}\in L(H_{\rho,0}(\mathbb{R},H))$ denoting
the canonical extensions of bounded linear operators on $H$ such
that
\begin{equation}
\begin{array}{l}
M_{0}\mbox{ is selfadjoint and strictly positive definite,}\\
M_{1}\mbox{ is skew-selfadjoint and }\\
M_{2}\mbox{ is selfadjoint.}
\end{array}\label{eq:pos-def}
\end{equation}
 As a simplification of notation we shall omit henceforth the closure
bar and simply write 
\begin{equation}
\left(\partial_{0}M_{0}+M_{1}+\partial_{0}^{-1}M_{2}+A\right)U=f\label{eq:evo3}
\end{equation}
to describe problems of the form (\ref{eq:evo2}).%
\footnote{This can be justified rigorously by employing the theory of Sobolev-lattices
(see \cite{pre05919306}).%
} 

We have singled out this rather special case%
\footnote{It should, however, be noted that by considering operator coefficients
we include painlessly the case of spatially non-local media, compare
e.g. \cite{zbMATH01817639}, in a unified setting.%
} with $M_{0},M_{1},M_{2}$ satisfying (\ref{eq:pos-def}) and $A$
being skew-selfadjoint, since it describes wave phenomena with a conservation
property, which will be at the heart of this paper. For sake of reference
we record the following well-posedness result, adapted to this simplified
situation.
\begin{thm}[{\cite[Solution Theory]{1200.35050}}]
\label{linear-evolutionary-solution-theory}Let (\ref{eq:pos-def})
hold. Then, for all sufficiently large%
\footnote{If $M_{2}$ is non-negative we can choose \emph{any }$\rho>0$.%
} $\rho\in\rga0$ and every $f\in H_{\rho,0}\left(\mathbb{R},H\right)$
there is a unique solution $U\in H_{\rho,0}\left(\mathbb{R},H\right)$
solving problem (\ref{eq:evo2}). Moreover, $\left(\partial_{0}M_{0}+M_{1}+\partial_{0}^{-1}M_{2}+A\right)^{-1}$
is a causal, continuous operator, where causality means that if $f=0$
on $\oi{-\infty}a$ then so is $U=\left(\partial_{0}M_{0}+M_{1}+\partial_{0}^{-1}M_{2}+A\right)^{-1}f$
, $a\in\mathbb{R}$.\end{thm}
\begin{rem}
\label{rem:regularity}Using the fact that the solution operator $(\partial_{0}M_{0}+M_{1}+\partial_{0}^{-1}M_{2}+A)^{-1}$
commutes with $\partial_{0}$, we get that the solution $U$ is as
regular as the given source term $f$, i.e. if $f\in D(\partial_{0}^{k})$
for some $k\in\mathbb{N}$ then so is $U$.
\end{rem}
It is clear that in the even simpler case with $M_{2}=0$, the fundamental
solution associated with (\ref{eq:evo2}) can be given in terms of
a function of the skew-selfadjoint operator $M_{0}^{-1/2}\left(M_{1}+A\right)M_{0}^{-1/2}$
as
\[
\left(\chi_{_{\lci0\infty}}\left(t\right)\: M_{0}^{-1/2}\exp\left(-t\: M_{0}^{-1/2}\left(M_{1}+A\right)M_{0}^{-1/2}\right)M_{0}^{1/2}\right)_{t\in\mathbb{R}}.
\]
For suitable $f$ the solution can of course be written as a convolution
integral of $f$ with this fundamental solution (as done in semi-group
theory). Although this may be interesting to note, we shall, however,
have no reason to use this fact in the following. As we shall see,
most of the applications we shall investigate have indeed $M_{2}=0$.
The only exception is the microstretch model discussed in Section
\ref{sec:Microstretch-Models}, which has indeed a more elaborate
conserved norm.

That the material law operator (\ref{eq:material}) describes a loss-less
situation is a consequence of $M_{1}$ being skew-selfadjoint and
$M_{2}$ being selfadjoint.
\begin{thm}
\label{thm:energy balance}Let $U,f\in H_{\rho,0}\left(\mathbb{R},H\right)$
such that 
\[
\left(\partial_{0}M_{0}+M_{1}+\partial_{0}^{-1}M_{2}+A\right)U=f.
\]
Then
\begin{multline*}
\frac{1}{2}\left\langle U|M_{0}U\right\rangle _{0}\left(b\right)+\frac{1}{2}\left\langle \partial_{0}^{-1}U|M_{2}\partial_{0}^{-1}U\right\rangle _{0}\left(b\right)\\
=\frac{1}{2}\left\langle U|M_{0}U\right\rangle _{0}\left(a\right)+\frac{1}{2}\left\langle \partial_{0}^{-1}U|M_{2}\partial_{0}^{-1}U\right\rangle _{0}\left(a\right)+\int_{a}^{b}\Re\left\langle U|f\right\rangle _{0}
\end{multline*}
for almost every $a,b\in\mathbb{R}$.\end{thm}
\begin{proof}
By the spectral theorem for bounded selfadjoint operators, we can
assume that $M_{2}$ is given as a multiplication operator with a
real-valued bounded function $V$ defined on some measure space. The
operators of multiplying with the positive part and negative part
of $V$ will respectively be denoted by $M_{2,+}$ and $M_{2,-}$.
It follows that $M_{2}=M_{2,+}-M_{2,-}$. 

We begin to prove the statement for $f\in D(\partial_{0}).$ Then
$U\in D(\partial_{0})$ according to Remark \ref{rem:regularity},
which in particular implies that $U$ is continuous by Sobolev's embedding
theorem (see \cite[Lemma 3.1.59]{pre05919306}). Furthermore, since
\[
AU=f-\partial_{0}M_{0}U-M_{1}U-\partial_{0}^{-1}M_{2}U\in H_{\rho,0}(\mathbb{R};H)
\]
we deduce that $U$ attains values in the domain of $A$. Thus, for
$a,b\in\mathbb{R}$ we compute 
\begin{align*}
 & \intop_{a}^{b}\Re\langle U(t)|f(t)\rangle\, dt\\
 & =\intop_{a}^{b}\Re\langle U(t)|\partial_{0}M_{0}U(t)\rangle\, dt+\intop_{a}^{b}\Re\langle U(t)|M_{1}U(t)\rangle\, dt+\intop_{a}^{b}\Re\langle U(t)|\partial_{0}^{-1}M_{2}U(t)\rangle\, dt\\
 & \quad+\intop_{a}^{b}\Re\langle U(t)|AU(t)\rangle\, dt\\
 & =\intop_{a}^{b}\Re\langle U(t)|\partial_{0}M_{0}U(t)\rangle\, dt+\intop_{a}^{b}\Re\langle U(t)|\partial_{0}^{-1}M_{2}U(t)\rangle\, dt,\\
 & =\intop_{a}^{b}\Re\langle U(t)|\partial_{0}M_{0}U(t)\rangle\, dt+\intop_{a}^{b}\Re\langle U(t)|\partial_{0}^{-1}M_{2,+}U(t)\rangle\, dt-\intop_{a}^{b}\Re\langle U(t)|\partial_{0}^{-1}M_{2,-}U(t)\rangle\, dt
\end{align*}
since $M_{1}$ and $A$ are skew-selfadjoint. Moreover, using $\partial_{0}\langle W|W\rangle(t)=2\Re\langle W(t)|\partial_{0}W(t)\rangle$
for $W\in D(\partial_{0})$ we get that 
\begin{align*}
 & \intop_{a}^{b}\Re\langle U(t)|f(t)\rangle\, dt\\
 & =\intop_{a}^{b}\Re\langle M_{0}^{\frac{1}{2}}U(t)|\partial_{0}M_{0}^{\frac{1}{2}}U(t)\rangle\, dt+\intop_{a}^{b}\Re\langle M_{2,+}^{\frac{1}{2}}U(t)|\partial_{0}^{-1}M_{2,+}^{\frac{1}{2}}U(t)\rangle\, dt\\
 & \quad-\intop_{a}^{b}\Re\langle M_{2,-}^{\frac{1}{2}}U(t)|\partial_{0}^{-1}M_{2,-}^{\frac{1}{2}}U(t)\rangle\, dt\\
 & =\frac{1}{2}\intop_{a}^{b}\partial_{0}\langle M_{0}^{\frac{1}{2}}U|M_{0}^{\frac{1}{2}}U\rangle(t)\, dt+\frac{1}{2}\intop_{a}^{b}\partial_{0}\langle\partial_{0}^{-1}M_{2,+}^{\frac{1}{2}}U|\partial_{0}^{-1}M_{2,+}^{\frac{1}{2}}U\rangle(t)\, dt\\
 & \quad-\frac{1}{2}\intop_{a}^{b}\partial_{0}\langle\partial_{0}^{-1}M_{2,-}^{\frac{1}{2}}U|\partial_{0}^{-1}M_{2,-}^{\frac{1}{2}}U\rangle(t)\, dt\\
 & =\frac{1}{2}\left(\langle M_{0}^{\frac{1}{2}}U|M_{0}^{\frac{1}{2}}U\rangle(b)-\langle M_{0}^{\frac{1}{2}}U|M_{0}^{\frac{1}{2}}U\rangle(a)\right)\\
 & \quad+\frac{1}{2}\left(\langle\partial_{0}^{-1}M_{2,+}^{\frac{1}{2}}U|\partial_{0}^{-1}M_{2,+}^{\frac{1}{2}}U\rangle(b)-\langle\partial_{0}^{-1}M_{2,+}^{\frac{1}{2}}U|\partial_{0}^{-1}M_{2,+}^{\frac{1}{2}}U\rangle(a)\right)\\
 & \quad-\frac{1}{2}\left(\langle\partial_{0}^{-1}M_{2,-}^{\frac{1}{2}}U|\partial_{0}^{-1}M_{2,-}^{\frac{1}{2}}U\rangle(b)-\langle\partial_{0}^{-1}M_{2,-}^{\frac{1}{2}}U|\partial_{0}^{-1}M_{2,-}^{\frac{1}{2}}U\rangle(a)\right)
\end{align*}
which yields the asserted energy balance equation. If $f\in H_{\rho,0}(\mathbb{R},H)$
we can approximate $f$ by a sequence $(f_{n})_{n\in\mathbb{N}}$
of functions in $D(\partial_{0}).$ Since the solution operator is
continuous, we obtain 
\[
U_{n}\coloneqq\left(\partial_{0}M_{0}+M_{1}+\partial_{0}^{-1}M_{2}+A\right)^{-1}f_{n}\to U.
\]
 According to what we have shown above, $U_{n}\in D(\partial_{0})$
satisfies the energy balance equation 
\begin{multline*}
\frac{1}{2}\left\langle U_{n}|M_{0}U_{n}\right\rangle _{0}\left(b\right)+\frac{1}{2}\left\langle \partial_{0}^{-1}U_{n}|M_{2}\partial_{0}^{-1}U_{n}\right\rangle _{0}\left(b\right)\\
=\frac{1}{2}\left\langle U_{n}|M_{0}U_{n}\right\rangle _{0}\left(a\right)+\frac{1}{2}\left\langle \partial_{0}^{-1}U_{n}|M_{2}\partial_{0}^{-1}U_{n}\right\rangle _{0}\left(a\right)+\int_{a}^{b}\Re\left\langle U_{n}|f_{n}\right\rangle _{0}.
\end{multline*}
Passing to an almost everywhere convergent subsequence, we derive
the energy balance for $U$ by letting $n$ tend to infinity in the
previous equality.\end{proof}
\begin{rem}
\label{rem:conservation_and_blocks} (a) Note that, if $f=0$ on an
interval $I\subseteq\mathbb{R}$ in the situation of the previous
theorem then 
\[
t\mapsto\frac{1}{2}\left\langle U|M_{0}U\right\rangle _{0}\left(t\right)+\frac{1}{2}\left\langle \partial_{0}^{-1}U|M_{2}\partial_{0}^{-1}U\right\rangle _{0}\left(t\right)
\]
is a constant function in $L^{\infty,\mathrm{loc}}\left(I\right)$.
For $M_{1}=M_{2}=0$ this results in $U=\left(\partial_{0}M_{0}+A\right)^{-1}f,$
and that 
\[
t\mapsto\frac{1}{2}\left\langle U|M_{0}U\right\rangle _{0}\left(t\right)=\frac{1}{2}\left|\sqrt{M_{0}}U\right|_{0}^{2}\left(t\right)
\]
is a constant function in $L^{\infty,\mathrm{loc}}\left(I\right)$. 

Using that $\sqrt{M_{0}}U$ satisfies the equation 
\[
\left(\partial_{0}+\sqrt{M_{0}}^{-1}A\sqrt{M_{0}}^{-1}\right)\sqrt{M_{0}}U=\sqrt{M_{0}}^{-1}f,
\]
we get that 
\[
\left(\partial_{0}+\sqrt{M_{0}}^{-1}A\sqrt{M_{0}}^{-1}\right)\left(\left(\sqrt{M_{0}}^{-1}A\sqrt{M_{0}}^{-1}\right)\sqrt{M_{0}}U\right)=\sqrt{M_{0}}^{-1}AM_{0}^{-1}f,
\]
provided that $U$ and $f$ are regular enough. Using that $\sqrt{M_{0}}^{-1}A\sqrt{M_{0}}^{-1}$
is again skew-selfadjoint, we obtain by Theorem \ref{thm:energy balance}
the preservation of the norm%
\footnote{Indeed, any Borel function $f$ of the skew-selfadjoint operator $\tilde{A}\coloneqq\sqrt{M_{0}}^{-1}A\sqrt{M_{0}}^{-1}$
in the sense of the associated function calculus would yield
\[
t\mapsto\frac{1}{2}\left|f\left(\tilde{A}\right)\sqrt{M_{0}}U\right|_{0}^{2}\left(t\right)
\]
 as a preserved quantity on $I$, provided that $\sqrt{M_{0}}U$ is
in the domain of $f\left(\tilde{A}\right)$.%
} 
\[
\frac{1}{2}\left|\left(\sqrt{M_{0}}^{-1}A\sqrt{M_{0}}^{-1}\right)\sqrt{M_{0}}U\right|_{0}^{2}
\]
on the interval $I$.

(b) In the situation of (a), we assume for sake of simplicity -- in
particular with respect to applications -- that $A$ is of the operator
block matrix form
\begin{equation}
A=\left(\begin{array}{cc}
0 & -G^{*}\\
G & 0
\end{array}\right),\label{eq:Hamiltonian}
\end{equation}
where $G:D\left(G\right)\subseteq H_{0}\to H_{1}$ is a densely defined,
closed, linear operator between two Hilbert spaces $H_{0},H_{1}$,
so that skew-selfadjointness of $A$ in $H\coloneqq H_{0}\oplus H_{1}$
is evident. Assuming in addition that $M_{0}$ can be written as a
block diagonal operator $\left(\begin{array}{cc}
M_{00} & 0\\
0 & M_{11}
\end{array}\right)$ with respect to the block structure induced by $A$, we obtain that
the function 
\begin{align*}
t & \mapsto\frac{1}{2}\left|\left(\begin{array}{cc}
0 & -\sqrt{M_{00}}^{-1}G^{*}\sqrt{M_{11}}^{-1}\\
\sqrt{M_{11}}^{-1}G\sqrt{M_{00}}^{-1} & 0
\end{array}\right)\left(\begin{array}{c}
\sqrt{M_{00}}U_{0}\\
\sqrt{M_{11}}U_{1}
\end{array}\right)\right|_{0}^{2}\left(t\right),\\
 & \mapsto\frac{1}{2}\left\langle \left.\left(\begin{array}{c}
-\sqrt{M_{00}}^{-1}G^{*}U_{1}\\
\sqrt{M_{11}}^{-1}GU_{0}
\end{array}\right)\right|\left(\begin{array}{c}
-\sqrt{M_{00}}^{-1}G^{*}U_{1}\\
\sqrt{M_{11}}^{-1}GU_{0}
\end{array}\right)\right\rangle _{0}\left(t\right),\\
 & \mapsto\frac{1}{2}\left\langle G^{*}U_{1}|M_{00}^{-1}G^{*}U_{1}\right\rangle _{0}^{2}\left(t\right)+\frac{1}{2}\left\langle GU_{0}|M_{11}^{-1}GU_{0}\right\rangle _{0}\left(t\right)
\end{align*}
is constant on $I$. Since $U=\left(\begin{array}{c}
U_{0}\\
U_{1}
\end{array}\right)$ satisfies 
\begin{align*}
\partial_{0}M_{00}U_{0}-G^{\ast}U_{1} & =0,\\
\partial_{0}M_{11}U_{1}+GU_{0} & =0,
\end{align*}
on $I$, this yields
\begin{align*}
t & \mapsto\frac{1}{2}\left\langle \partial_{0}M_{00}U_{0}|M_{00}^{-1}\partial_{0}M_{00}U_{0}\right\rangle _{0}\left(t\right)+\frac{1}{2}\left\langle GU_{0}|M_{11}^{-1}GU_{0}\right\rangle _{0}\left(t\right),\\
 & \mapsto\frac{1}{2}\left\langle \partial_{0}U_{0}|M_{00}\partial_{0}U_{0}\right\rangle _{0}\left(t\right)+\frac{1}{2}\left\langle GU_{0}|M_{11}^{-1}GU_{0}\right\rangle _{0}\left(t\right),
\end{align*}
as a preserved quantity on the interval $I$, an expression, which
-- with the notable exception of Maxwell's system of electrodynamics
-- is frequently used to state conservation of energy, rather than
the simple norm preservation.
\end{rem}

\subsection{A Mechanism Deriving Evolutionary Equations from Given Ones\label{sub:A-Mechanism-Deriving}}

We start this section by giving a typical example for $A$ admitting
the block structure as stated in Remark \ref{rem:conservation_and_blocks}:
\begin{example}
\label{example:covaraint derivative}In \cite{Picard201354} a particular
case for the skew-selfadjoint operator $A=\left(\begin{array}{cc}
0 & -G^{*}\\
G & 0
\end{array}\right)$ has been considered, which, if Dirichlet type boundary conditions
are to be imposed, leads to 
\[
G=\interior\nabla
\]
on $\bigoplus_{q\in\mathbb{N}}L^{2,q}\left(\Omega\right)$, where
$L^{2,q}\left(\Omega\right)$ denotes the Hilbert space of covariant
tensors of degree $q\in\mathbb{N}$, defined on an arbitrary non-empty
open subset $\Omega$ of a Riemannian submanifold $(M,g)$. The operator
$\interior\nabla$ is defined as the closure of the covariant derivative
acting on smooth covariant tensors with compact support in $\Omega$.
The inner product of $L^{2,q}\left(\Omega\right)$ is given by
\[
\left(\phi,\psi\right)\mapsto\int_{\Omega}\left\langle \phi|\psi\right\rangle _{\otimes q}\: dV,
\]
where $V$ is the volume element of the Riemannian manifold $M$ and
$\left\langle \:\cdot\:|\:\cdot\:\right\rangle _{\otimes q}$ denotes
the inner product for covariant $q$-tensors induced by the Riemannian
metric tensor $g$. The divergence operator $\nabla\cdot$ (or $\dive$)
is defined as the skew-adjoint of $\interior\nabla$, i.e.
\[
-\nabla\cdot\coloneqq\left(\interior\nabla\right)^{*}=G^{\ast}.
\]
We shall focus here on the Dirichlet case, although it should be clear
that many other boundary conditions can be treated in an analogous
way by prescribing a dense domain for $\nabla$ to establish $G$
($\nabla$ is the skew-adjoint of an analogously defined $\interior\nabla\cdot$).
If for example the domain of $\nabla$ is not constrained at all we
have Neumann type boundary conditions induced via $G^{*}=\interior\nabla\cdot$~.
However, to make the presentation not unnecessarily difficult to follow,
we shall focus on the Dirichlet case throughout.

Considering elements in $\bigoplus_{q\in\mathbb{N}}L^{2,q}\left(\Omega\right)$
as infinite column vectors we see that $A$ may also be written in
the suggestive infinite tridiagonal block operator matrix form
\[
A=\left(\begin{array}{ccccc}
0 & \;\nabla\cdot & \,0 & \cdots & \cdots\\
\interior\nabla & \;0 & \;\nabla\cdot & 0\\
0 & \;\interior\nabla & \;0 & \ddots & \quad\ddots\\
\vdots & \;0 & \ddots & \ddots & \;\ddots\\
\vdots &  & \;\,\ddots & \;\ddots & \;\ddots
\end{array}\right).
\]
From the associated evolutionary operators
\[
\left(\partial_{0}M_{0}+M_{1}+\partial_{0}^{-1}M_{2}+A\right)
\]
for suitable $M_{0},M_{1},M_{2}$ desired parts may be extracted from
this ``mother'' class of operators by a rigorous mechanism discussed
in \cite{Picard201354}. \end{example}
\begin{defn}
\label{def:relatives}Let $H,X$ be Hilbert spaces and $A:D\left(A\right)\subseteq H\to H$
densely defined closed linear. Moreover, let $B:H\to X$ be linear
and such that:
\begin{itemize}
\item $AB^{*}$ is densely defined (in $X$).
\end{itemize}
Then we call $B$ \emph{compatible }with $A$ and $\overline{BA}B^{*}$
the $\left(B\right)$-\emph{relative} (or simply a \emph{relative)}
of $A$. If the mapping $B$ is not a bijection, then we call $\overline{BA}B^{*}$
the $\left(B\right)$-\emph{descendant} (or simply a \emph{descendant)}
of $A$ (and $A$ the \emph{mother }operator of $\overline{BA}B^{*}$).
\end{defn}
We recall from \cite{Picard201354} the following result.
\begin{thm}
\label{thm:compatible}Let $A:D(A)\subseteq H_{0}\to H_{1}$ a densely
defined closed linear operator and $B\in L(H_{0},X)$ such that $AB^{\ast}$
is densely defined. Then

\[
(AB^{\ast})^{\ast}=\overline{BA^{\ast}}.
\]
\end{thm}
\begin{proof}
Obviously, $BA^{\ast}\subseteq\left(AB^{\ast}\right)^{\ast}$ which
yields $\overline{BA^{\ast}}\subseteq\left(AB^{\ast}\right)^{\ast}.$
Let now $x\in D\left(\left(BA^{\ast}\right)^{\ast}\right).$ Then
for each $y\in D(A^{\ast})$ we compute 
\[
\langle A^{\ast}y|B^{\ast}x\rangle=\langle BA^{\ast}y|x\rangle=\langle y|\left(BA^{\ast}\right)^{\ast}x\rangle
\]
which yields $x\in D(AB^{\ast})$, i.e. $\left(BA^{\ast}\right)^{\ast}\subseteq AB^{\ast}.$
The latter gives $\left(AB^{\ast}\right)^{\ast}\subseteq\left(BA^{\ast}\right)^{\ast\ast}=\overline{BA^{\ast}}$,
which completes the proof.\end{proof}
\begin{cor}
Let $G:D(G)\subseteq H_{0}\to H_{1}$ a densely defined closed linear
operator. Moreover, let $B_{0}\in L(H_{0},X)$ and $B_{1}\in L(H_{1},Y)$
such that $GB_{0}^{\ast}$ and $G^{\ast}B_{1}^{\ast}$ are densely
defined. Then 
\[
\left(\overline{B_{1}G}B_{0}^{\ast}\right)^{\ast}=\overline{B_{0}G^{\ast}B_{1}^{\ast}}
\]
\end{cor}
\begin{proof}
By Theorem \ref{thm:compatible} we have that 
\begin{align*}
\left(\overline{B_{1}G}B_{0}^{\ast}\right)^{\ast} & =\left(\left(G^{\ast}B_{1}^{\ast}\right)^{\ast}B_{0}^{\ast}\right)^{\ast}\\
 & =\left(B_{0}G^{\ast}B_{1}^{\ast}\right)^{\ast\ast}\\
 & =\overline{B_{0}G^{\ast}B_{1}^{\ast}}.\tag*{\qedhere}
\end{align*}

\end{proof}
As an immediate consequence for the type of equations of interest
here, we record for later reference the following corollary.
\begin{cor}
\label{cor:well-posedness_descendant}Let $G:D(G)\subseteq H_{0}\to H_{1}$
a densely defined closed linear operator and set $A\coloneqq\left(\begin{array}{cc}
0 & -G^{\ast}\\
G & 0
\end{array}\right)$ and $H\coloneqq H_{0}\oplus H_{1}$. Moreover, let $M_{0},M_{1},M_{2}:H\to H$
be continuous linear operators satisfying (\ref{eq:pos-def}) and
$S=\left(\begin{array}{cc}
B_{0} & 0\\
0 & B_{1}
\end{array}\right)\in L(H,Z),$ where $B_{0}\in L(H_{0},X),\, B_{1}(H_{1},Y)$ and $Z\coloneqq X\oplus Y$
such that $GB_{0}^{\ast}$ and $G^{\ast}B_{1}^{\ast}$ are densely
defined. Then $A_{S}\coloneqq\left(\begin{array}{cc}
0 & -\overline{B_{0}G^{\ast}}B_{1}^{\ast}\\
\overline{B_{1}GB_{0}^{\ast}} & 0
\end{array}\right)$ and $M_{1,S}\coloneqq SM_{1}S^{\ast}$ are skew-selfadjoint, while
the operators $M_{i,S}\coloneqq SM_{i}S^{\ast}\in L(Z)$, $i\in\{0,2\}$
are selfadjoint. If $S^{\ast}$ has a bounded left-inverse%
\footnote{This assumption is met for instance if $S$ is the orthogonal projection
onto a closed subspace $Z$ of $H$. In this case $S^{\ast}$ is the
canonical embedding of $Z$ into $H$ and $\left(S^{\ast}\right)^{-1}=S|_{S^{\ast}[H]}$,
compare Lemma \ref{fac-lem}. %
}, i.e. there exists $\left(S^{\ast}\right)^{-1}:S^{\ast}[Z]\subseteq H\to Z$
with $\left(S^{\ast}\right)^{-1}S^{\ast}=1_{Z}$, then $M_{i,S}$
satisfies the hypotheses (\ref{eq:pos-def}), $i\in\{0,1,2\}$, and
in this case the evolutionary problem 
\[
\left(\partial_{0}M_{0,S}+M_{1,S}+\partial_{0}^{-1}M_{2,S}+A_{S}\right)U=F
\]
is well-posed in $H_{\rho,0}\left(\mathbb{R},Z\right)$ for all sufficiently
large $\rho\in\oi0\infty$ and the corresponding solution operator
is causal. \end{cor}
\begin{rem}
\label{rem:closure}(a) Note that, in general, the operator $A_{S}$
is not the $\left(S\right)$-descendant of $A,$ since $A_{S}\ne\overline{SA}S^{\ast}.$
In order to replace $A_{S}$ by the $\left(S\right)$-descendant of
$A$ in the previous corollary, we have to guarantee the skew-selfadjointness
of 
\[
\overline{SA}S^{\ast}=\left(\begin{array}{cc}
0 & -\overline{B_{0}G^{\ast}}B_{1}^{\ast}\\
\overline{B_{1}G}B_{0}^{\ast} & 0
\end{array}\right),
\]
that is, we have to guarantee 
\begin{equation}
\overline{B_{1}G}B_{0}^{\ast}=\overline{B_{1}GB_{0}^{\ast}}.\label{eq:closure}
\end{equation}
The latter is true, if one for instance requires that $B_{0}^{\ast}:X\to H_{0}$
has a bounded left-inverse, i.e. there exists $C:\overline{B_{0}^{\ast}[X]}\to X$
such that $CB_{0}^{\ast}=1_{X},$ the identity on $X$, and that the
set $B_{0}^{\ast}[X]\cap D(G)$ is a core for $\overline{B_{1}G}|_{\overline{B_{0}^{\ast}[X]}}$.
Indeed, let $x\in D(\overline{B_{1}G}B_{0}^{\ast}).$ Then by assumption
there exists a sequence $(x_{n})_{n\in\mathbb{N}}$ in $X$ such that
$B_{0}^{\ast}x_{n}\in D(G)$ and
\begin{align*}
B_{0}^{\ast}x_{n} & \to B_{0}^{\ast}x\\
B_{1}GB_{0}^{\ast}x_{n} & \to\overline{B_{1}G}B_{0}^{\ast}x
\end{align*}
as $n\to\infty.$ Using that $B_{0}^{\ast}$ has a bounded left-inverse,
we obtain $x_{n}\to x$, which yields $x\in D\left(\overline{B_{1}GB_{0}^{\ast}}\right).$
Thus, $\overline{B_{1}G}B_{0}^{\ast}\subseteq\overline{B_{1}GB_{0}^{\ast}}$
and since the other inclusion holds trivially, we derive the assertion.\\
(b) Note that (\ref{eq:closure}) is false in general, \cite{2013}.
The counterexample is based on the following observation made by J.
Epperlein and H. Vogt:\\
We ask the following question: Is there a closable operator $A$ and
a bounded operator $B$ such that $\overline{AB}\subsetneq\overline{A}B$?
Let $H$ be an infinite dimensional Hilbert space and $A_{0},A_{1}$
be two closed, densely defined operators on $H,$ such that $A_{0}\subsetneq A_{1}.$
Moreover let $(y_{n})_{n\in\mathbb{N}}$ a linear independent sequence
in $H$ such that $y_{n}\to0$. For a fixed $x_{1}\in D(A_{1})\setminus D(A_{0})$
we define the operator 
\begin{align*}
A:D(A)\subseteq H\oplus H & \to H\\
(x,y) & \mapsto A_{1}x,
\end{align*}
where 
\[
D(A)\coloneqq\left(D(A_{0})\times\{0\}\right)\oplus\lspan\left\{ \left.(x_{1},y_{n})\,\right|\, n\in\mathbb{N}\right\} .
\]
Moreover, set 
\begin{align*}
B:H & \to H\oplus H\\
x & \mapsto(x,0).
\end{align*}
Then $AB=A_{0},$ while $x_{1}\in D(\overline{A}B).$\end{rem}
\begin{example}
We want to give an illustrative example for the aforementioned mechanism.
For this recall the differential geometric setting in Example \ref{example:covaraint derivative}.
Set $H_{0}\coloneqq\bigoplus_{q=0}^{\infty}L^{2,q}(\Omega)$ and $H_{1}\coloneqq\bigoplus_{q=1}^{\infty}L^{2,q}(\Omega).$
Define $G\coloneqq\interior\nabla,$ which yields $G^{\ast}=-\nabla\cdot$,
which are both densely defined closed linear operators. $A$ is given
by 
\[
A\coloneqq\left(\begin{array}{cc}
0 & \nabla\cdot\\
\interior\nabla & 0
\end{array}\right):D(A)\subseteq H\to H,
\]
where $H\coloneqq H_{0}\oplus H_{1}.$ Let 
\begin{align*}
B_{0}:H_{0} & \to L^{2,1}(\Omega)\\
\left(\begin{array}{c}
f_{0}\\
f_{1}\\
f_{2}\\
\vdots
\end{array}\right) & \mapsto f_{1}
\end{align*}
and 
\begin{align*}
B_{1}:H_{1} & \to L^{2,2}(\Omega)\\
\left(\begin{array}{c}
f_{1}\\
f_{2}\\
\vdots
\end{array}\right) & \mapsto f_{2}
\end{align*}
be the canonical projections onto the second entries in their respective
domain spaces. Then $B_{0}^{\ast}$ and $B_{1}^{\ast}$ are just the
corresponding canonical embeddings. Thus, $\overline{B_{1}\interior\nabla B_{0}^{\ast}}\subseteq L^{2,1}(\Omega)\oplus L^{2,2}(\Omega)$
is just the covariant derivative on $1$-tensor fields with Dirichlet
boundary conditions%
\footnote{Indeed, in this situation $\overline{B_{1}\interior\nabla B_{0}^{\ast}}=\overline{B_{1}\interior\nabla}B_{0}^{\ast}$
according to Remark \ref{rem:closure}.%
} and hence, $\overline{B_{0}\nabla\cdot\,}B_{1}^{\ast}\subseteq L^{2,2}(\Omega)\oplus L^{2,1}(\Omega)$
equals the usual divergence on $2$-tensor fields without boundary
conditions. Thus, the descendant problem 
\[
\left(\partial_{0}M_{0,S}+M_{1,S}+\partial_{0}^{-1}M_{2,S}+A_{S}\right)U=F
\]
in Corollary \ref{cor:well-posedness_descendant} is just a reduction
of the full problem to a closed subspace, namely $L^{2,1}(\Omega)\oplus L^{2,2}(\Omega).$
\end{example}
Although the example above is trivial, it turns out that in many application
this process of dimension reduction occurs, where the operators $B_{0}$
and $B_{1}$ are given as suitable orthogonal projections. For the
purpose of our following considerations we first record the following
rather elementary observation, which is just a variant of the projection
theorem.
\begin{lem}[{see e.g.~\cite[Lemma 3.2]{PTW_frac}}]
\label{fac-lem}Let $\iota$ be the canonical (isometric) embedding
of a closed subspace $V\subseteq H$ into $H.$ Then $\iota\iota^{*}$
is the orthogonal projector onto $V$. Let $\kappa$ be the canonical
embedding of $V^{\perp}$ into $H$ then we have
\begin{equation}
\iota\iota^{*}+\kappa\kappa^{*}=1.\label{eq:pro--sum}
\end{equation}
\end{lem}
\begin{rem}
\label{rem-block}Equality (\ref{eq:pro--sum}) may be written in
an intuitive block operator matrix notation as
\[
\left(\begin{array}{cc}
\iota & \kappa\end{array}\right)\left(\begin{array}{c}
\iota^{*}\\
\kappa^{*}
\end{array}\right)=1.
\]
It may also be worth noting that $\iota^{*}\iota:V\to V$ and $\kappa^{*}\kappa:V^{\perp}\to V^{\perp}$
are just the identities on $V$ and $V^{\perp}$, respectively. It
is common practice to identify $H$ and $V\oplus V^{\perp}$, which
makes
\begin{equation}
\left(\begin{array}{c}
\iota^{*}\\
\kappa^{*}
\end{array}\right):H\to V\oplus V^{\perp}\eqqcolon\left(\begin{array}{c}
V\\
V^{\perp}
\end{array}\right),\; x\mapsto\left(\begin{array}{c}
\iota^{*}\\
\kappa^{*}
\end{array}\right)x=\left(\begin{array}{c}
\iota^{*}x\\
\kappa^{*}x
\end{array}\right)\label{eq:proj}
\end{equation}
the identity. However, for our purposes it appears helpful to avoid
this identification. 

The mapping (\ref{eq:proj}) is obviously unitary, which allows us
for example to study an equation of the form 
\[
NU=F
\]
for a bounded linear operator $N$ in $H$ via the unitarily equivalent
block operator matrix equation 
\[
\left(\begin{array}{c}
\iota^{*}\\
\kappa^{*}
\end{array}\right)N\left(\begin{array}{cc}
\iota & \kappa\end{array}\right)\left(\begin{array}{c}
\iota^{*}\\
\kappa^{*}
\end{array}\right)U=\left(\begin{array}{cc}
\iota^{*}N\iota & \iota^{*}N\kappa\\
\kappa^{*}N\iota & \kappa^{*}N\kappa
\end{array}\right)\left(\begin{array}{c}
\iota^{*}U\\
\kappa^{*}U
\end{array}\right)=\left(\begin{array}{c}
\iota^{*}F\\
\kappa^{*}F
\end{array}\right).
\]
Note that
\[
\left(\begin{array}{cc}
\iota & \kappa\end{array}\right):\left(\begin{array}{c}
V\\
V^{\perp}
\end{array}\right)\to H,\;\left(\begin{array}{c}
u\\
v
\end{array}\right)\mapsto\left(\begin{array}{cc}
\iota & \kappa\end{array}\right)\left(\begin{array}{c}
u\\
v
\end{array}\right)=\iota u+\kappa v
\]
is the inverse of (\ref{eq:proj}). 
\end{rem}
As a general notational convention we shall use that if $P:H\to H$
is an orthogonal projector then the canonical embedding of its range
into the Hilbert space $H$ will be denoted by $\iota_{P}$ so that
\[
P=\iota_{P}\iota_{P}^{*}.
\]

\begin{rem}
\label{rem:degenerate-materials}We give another typical example for
the reduction process via orthogonal projectors, arising due to singularities
in the constitutive relation. In applications one may find a formal
equation of the form 
\[
\partial_{0}V+\left(M_{1}+A\right)U=f,
\]
with a corresponding material relation, which is frequently simply
given as 
\[
U=N_{0}V,
\]
where $N_{0}\in L(H)$ is a selfadjoint, non-negative operator. If
$N_{0}$ is strictly positive definite then 
\[
V=M_{0}U
\]
with $M_{0}=N_{0}^{-1}$, in which case we are led to the particular
situation of evolutionary equations considered here. 

Frequently, however, the operator $N_{0}$ may be initially not invertible
since $N_{0}$ has a non-trivial null space. Let $Q_{0}$ be the (non-trivial)
orthogonal projector onto the range of $N_{0}$ $(Q_{0}\notin\{0,1\})$
then $1-Q_{0}$ is the (non-trivial) orthogonal projector onto the
kernel of $N_{0}$. This fact suggests to reduce the material relation
to 
\begin{align*}
\iota_{Q_{0}}^{*}U & =\left(\iota_{Q_{0}}^{*}N_{0}\iota_{Q_{0}}\right)\iota_{Q_{0}}^{*}V
\end{align*}
and assuming that $N_{0}$ is strictly positive definite on its own
range, i.e. $\iota_{Q_{0}}^{*}N_{0}\iota_{Q_{0}}$is strictly positive
definite. We have
\[
\iota_{Q_{0}}^{*}V=\left(\iota_{Q_{0}}^{*}N_{0}\iota_{Q_{0}}\right)^{-1}\left(\iota_{Q_{0}}^{*}U\right).
\]
The resulting evolutionary equation of the form assumed in this paper
is now the reduced equation
\begin{equation}
\left(\partial_{0}\tilde{M}_{0}+\tilde{M}_{1}+\tilde{A}\right)\tilde{U}=\tilde{f},\label{eq:descendant}
\end{equation}
where $\tilde{U}\coloneqq\iota_{Q_{0}}^{*}U$, $\tilde{A}\coloneqq\overline{\iota_{Q_{0}}^{*}A}\iota_{Q_{0}}$,
$\tilde{M}_{0}\coloneqq\left(\iota_{Q_{0}}^{*}N_{0}\iota_{Q_{0}}\right)^{-1}$,
$\tilde{M}_{1}\coloneqq\iota_{Q_{0}}^{*}M_{1}\iota_{Q_{0}}$, $\tilde{f}\coloneqq\iota_{Q_{0}}^{*}f$,
which is a descendant problem in the sense of Corollary \ref{cor:well-posedness_descendant},%
\footnote{In fact this is a well-posed descendant of a possibly ill-posed problem,
a situation which was not addressed in Corollary \ref{cor:well-posedness_descendant}.
However, if one replaces $N_{0}$ by the strictly positive definite
and selfadjoint operator $N_{0}+\epsilon\left(1-Q_{0}\right)$ for
some $\varepsilon>0$, the original problem is well-posed and its
descendant problem is actually given by (\ref{eq:descendant}) and
so Corollary \ref{cor:well-posedness_descendant} applies.%
} where $S=\iota_{Q_{0}}^{\ast}$, provided that $\tilde{A}$ is skew-selfadjoint
(compare Remark \ref{rem:closure}). In other words we are led to
seek solutions in the smaller space $H_{\rho,0}\left(\mathbb{R},Q_{0}\left[H\right]\right)$.
Note that $\tilde{M}_{0}$ is now by construction and the assumptions
on $N_{0}$ continuous selfadjoint and strictly positive definite
on $Q_{0}\left[H\right]=N_{0}\left[H\right]$ and with $M_{1}$ also
$\tilde{M}_{1}$ remains skew-selfadjoint. 
\end{rem}

\section{\label{sec:Cosserat-Elasticity}On Some Models of Deformable Solids
and their Interconnection}

We discuss the equations of elasticity in a $3$-dimensional (differentiable)
Riemannian submanifold $M$ and think of $\Omega\subseteq M$ being
an open subset of $M$ modeling the body under consideration in its
non-deformed state. Recall the functional analytic setting of Example
\ref{example:covaraint derivative}, i.e., $\interior\nabla$ denotes
the covariant derivative on $L^{2,s}(\Omega)$ taking values in $L^{2,s+1}(\Omega)$,
$s\in\mathbb{N}$, with (generalized) Dirichlet boundary conditions
and its skew-adjoint $\nabla\cdot$. It will be the purpose of the
following to discuss several models of elasticity and to describe
their interconnection. We show that independent of the physical interpretation
of the given quantities, it is possible invoking the mother and descendant
mechanism to derive all these models from the model for micromorphic
media proposed by R.D.~Mindlin, \cite{0119.40302}. Anticipating
the theory discussed in the previous sections and realizing that in
the theory of elasticity, one is confronted with symmetric, skew-symmetric
and trace-free parts of $2$-tensors, we introduce some projections
in $L^{2,2}(\Omega)$. 

For this let $\tau\in L^{2,2}(\Omega)$ be a covariant tensor field
of order $2$. For%
\footnote{We denote the tangent space at $p$ by 
\[
T_{p}M\coloneqq\mbox{span}\{\partial_{j}\phi(\phi^{-1}(p));\phi\colon U\subseteq\mathbb{R}^{3}\to M\cap\Omega\mbox{ local parametrization},j\in\{1,2,3\}\}.
\]
} $p\in\Omega$, $x,y\in T_{p}M$, we define 
\begin{align*}
\left(\mathrm{skew}\tau\right)_{p}\left(x,y\right) & \coloneqq\frac{1}{2}\left(\tau_{p}\left(x,y\right)-\tau_{p}\left(y,x\right)\right)\\
\left(\mathrm{sym}\tau\right)_{p}\left(x,y\right) & \coloneqq\frac{1}{2}\left(\tau_{p}\left(x,y\right)+\tau_{p}\left(y,x\right)\right)
\end{align*}
and 
\[
\trace:L^{2,2}(\Omega)\to L^{2,0}(\Omega)
\]
with 
\[
\left(\trace\tau\right)_{p}\coloneqq\left\langle \left\langle \:\cdot\:|\:\cdot\:\right\rangle _{p}|\tau_{p}\right\rangle _{\otimes2},
\]
where for the Riemannian metric tensor $g$ at $p$ we write $g_{p}=\left\langle \:\cdot\:|\:\cdot\:\right\rangle _{p}$.
It is 
\[
\left(\trace^{*}\varphi\right)_{p}=\varphi\left(p\right)g_{p}
\]
for $\varphi\in L^{2,0}(\Omega)$ and we set 
\[
\mathbb{P}=\frac{1}{3}\trace^{*}\trace.
\]
We define $\sym_{0}\coloneqq\left(1-\mathbb{P}\right)\mathrm{sym}=\mathrm{sym}\left(1-\mathbb{P}\right)$.
The operators $\sym,\sym_{0},\skew$ and $\mathbb{P}$ are orthogonal
projectors in $L^{2,2}\left(\Omega\right)$, which satisfy
\begin{equation}
\mathrm{skew}+\mathrm{sym}_{0}+\mathbb{P}=1.\label{eq:orthogonal-decomposition}
\end{equation}
The part $\mathrm{skew}+\mathrm{sym}_{0}=1-\mathbb{P}$ is frequently
referred to as \emph{deviatoric} part and $\mathbb{P}$ as \emph{volumetric}
or pressure part. The symmetric part is given by $\mathrm{sym}=\mathrm{sym}_{0}+\mathbb{P}=1-\mathrm{skew}$.
Recall our convention $\iota_{\sym}\colon\sym[L^{2,2}(\Omega)]\to L^{2,2}(\Omega),\phi\mapsto\phi$
to be the canonical injection and similarly for $\skew,\sym_{0}$
and $\mathbb{P}$.

Consider the particular Euclidean case with $\Omega$ a non-empty
open subset of $M=\mathbb{R}^{3}$ with the Euclidean inner product.
In this case
\[
T_{p}M=\mathbb{R}^{3}
\]
and the matrix representation of $g_{p}$ with respect to the canonical
basis of $\mathbb{R}^{3}$ is simply the unit matrix independent of
$p\in\Omega$. Thus, in particular $\trace\tau$ is just the matrix
trace of the matrix representation $\left(\tau_{ij}\right)_{i,j}$
of $\tau$ and the inner product of 2-tensors is the Frobenius inner
product
\[
\left(\tau,\sigma\right)\mapsto\trace\left(\left(\tau_{ij}\right)_{i,j}^{*}\left(\sigma_{kl}\right)_{k,l}\right).
\]

\subsection{\label{sub:Micromorphic-Media}Micromorphic Media}

R. D. Mindlin, \cite{0119.40302}, has proposed modified Cosserat
type media, which Eringen, see e.g. \cite{zbMATH03352829}, in his
comprehensive and systematic studies of generalized continuum mechanics
labeled as ``micromorphic''. We set-up the system properly in $H_{\rho,0}\left(\mathbb{R},H\right)$
with the underlying Hilbert space $H$ being
\[
H=L^{2,1}\left(\Omega\right)\oplus L^{2,2}\left(\Omega\right)\oplus L^{2,2}\left(\Omega\right)\oplus L^{2,3}\left(\Omega\right)\oplus\mathrm{sym}\left[L^{2,2}\left(\Omega\right)\right].
\]
The dynamic equations read as
\begin{align}
\partial_{0}\rho_{0}\dot{u}-\nabla\cdot\left(\iota_{\mathrm{sym}}\tau+\sigma\right) & =f\nonumber \\
\partial_{0}\rho_{2}\dot{\psi}-\nabla\cdot\mu-\sigma & =h,\label{eq:dyn_micro_med}
\end{align}
where $f\in H_{\rho,0}(\mathbb{R};L^{2,1}(\Omega)),h\in H_{\rho,0}(\mathbb{R};L^{2,2}(\Omega))$
are given quantities and $\rho_{0},\rho_{2}$ are the canonical extensions
of bounded linear operators within $L^{2,1}(\Omega)$ and $L^{2,2}(\Omega)$,
respectively. The unknowns are $\dot{u}\in H_{\rho,0}(\mathbb{R};L^{2,1}(\Omega)),\dot{\psi},\sigma\in H_{\rho,0}(\mathbb{R};L^{2,2}(\Omega)),\tau\in H_{\rho,0}(\mathbb{R};\mathrm{sym}[L^{2,2}(\Omega)])$
and $\mu\in H_{\rho,0}(\mathbb{R};L^{2,3}(\Omega)).$ The equations
(\ref{eq:dyn_micro_med}) are completed by the relations 
\begin{equation}
\left(\begin{array}{c}
\tau\\
\sigma\\
\mu
\end{array}\right)=\left(\begin{array}{ccc}
C_{0} & G_{0} & F_{0}\\
G_{0}^{*} & C_{1} & D_{0}\\
F_{0}^{*} & D_{0}^{*} & C_{2}
\end{array}\right)\left(\begin{array}{c}
\epsilon\\
\gamma\\
\kappa
\end{array}\right)\in H_{\rho,0}(\mathbb{R};\mathrm{sym}\left[L^{2,2}\left(\Omega\right)\right]\oplus L^{2,2}\left(\Omega\right)\oplus L^{2,3}\left(\Omega\right)),\label{eq:pre-matter}
\end{equation}

with 
\[
\epsilon\coloneqq\iota_{\mathrm{sym}}^{\ast}\nabla u,\:\kappa\coloneqq\nabla\psi,\:\gamma\coloneqq\nabla u-\psi
\]
for suitable bounded linear operators $C_{0},G_{0},F_{0},C_{1},D_{0},C_{2}$
only depending on the spatial variables%
\footnote{The dynamic equations already indicate that $\iota_{\mathrm{sym}}\tau+\sigma$
is most likely a more suitable unknown rather than both $\tau$ and
$\sigma$. Note that from $\iota_{\mathrm{sym}}\tau+\sigma$ and $\iota_{\mathrm{sym}}\sigma$
both can, however, be recovered. Since $\tau\in\mathrm{sym}\left[L^{2,2}\left(\Omega\right)\right]$
we have
\begin{align*}
\sigma & =\iota_{\mathrm{sym}}\iota_{\mathrm{sym}}^{*}\sigma+\mathrm{skew}\left(\iota_{\mathrm{sym}}\tau+\sigma\right),\\
\tau & =\iota_{\mathrm{sym}}^{*}\left(\iota_{\mathrm{sym}}\tau+\sigma\right)-\iota_{\mathrm{sym}}^{*}\sigma.
\end{align*}
} and $u=\partial_{0}^{-1}\dot{u}$, $\psi=\partial_{0}^{-1}\dot{\psi}$.
We want to give well-posedness conditions for the operators involved
of finding $(\dot{u},\dot{\psi})\in H_{\rho,0}(\mathbb{R};L^{2,1}(\Omega)\oplus L^{2,2}(\Omega))$
solving (\ref{eq:dyn_micro_med}) subject to (\ref{eq:pre-matter}).
In order to do so, we reformulate the problem. For this we consider
the block operator matrix
\begin{align*}
 & \left(\begin{array}{ccc}
\iota_{\mathrm{sym}}C_{0}\iota_{\mathrm{sym}}^{*}+G_{0}^{*}\iota_{\mathrm{sym}}^{*}+\iota_{\mathrm{sym}}G_{0}+C_{1} & \quad\iota_{\mathrm{sym}}F_{0}+D_{0} & \quad\iota_{\mathrm{sym}}G_{0}\iota_{\mathrm{sym}}+C_{1}\iota_{\mathrm{sym}}\\
F_{0}^{*}\iota_{\mathrm{sym}}^{*}+D_{0}^{*} & C_{2} & D_{0}^{*}\iota_{\mathrm{sym}}\\
\iota_{\mathrm{sym}}^{*}G_{0}^{*}\iota_{\mathrm{sym}}^{*}+\iota_{\mathrm{sym}}^{*}C_{1} & \iota_{\mathrm{sym}}^{*}D_{0} & \iota_{\mathrm{sym}}^{*}C_{1}\iota_{\mathrm{sym}}
\end{array}\right)\\
 & \in L(L^{2,2}\left(\Omega\right)\oplus L^{2,3}\left(\Omega\right)\oplus\mathrm{sym}\left[L^{2,2}\left(\Omega\right)\right]),
\end{align*}
which we assume to be continuously invertible throughout. We denote
its inverse by $W=\left(\begin{array}{ccc}
W_{00} & W_{01} & W_{02}\\
W_{01}^{\ast} & W_{11} & W_{21}\\
W_{02}^{\ast} & W_{21}^{\ast} & W_{22}
\end{array}\right)$.
\begin{thm}
\label{thm:micro-med-reform}Let $\rho\in\mathbb{R}_{>0}$. Then $(\dot{u},\dot{\psi})\in H_{\rho,0}(\mathbb{R};L^{2,1}(\Omega)\oplus L^{2,2}(\Omega))$
is a solution of (\ref{eq:dyn_micro_med}) subject to (\ref{eq:pre-matter})
if and only if $\left(\begin{array}{c}
\dot{u}\\
\dot{\psi}\\
\iota_{\mathrm{sym}}\tau+\sigma\\
\mu\\
\iota_{\mathrm{sym}}^{*}\sigma
\end{array}\right)\in H_{\rho,0}(\mathbb{R};L^{2,1}\left(\Omega\right)\oplus L^{2,2}\left(\Omega\right)\oplus L^{2,2}\left(\Omega\right)\oplus L^{2,3}\left(\Omega\right)\oplus\mathrm{sym}\left[L^{2,2}\left(\Omega\right)\right])$ solves
\[
\left(\partial_{0}M_{0}+M_{1}+A\right)\left(\begin{array}{c}
\dot{u}\\
\dot{\psi}\\
\iota_{\mathrm{sym}}\tau+\sigma\\
\mu\\
\iota_{\mathrm{sym}}^{*}\sigma
\end{array}\right)=\left(\begin{array}{c}
f\\
h\\
0\\
0\\
0
\end{array}\right),
\]
where 
\[
M_{0}=\left(\begin{array}{ccc}
\rho_{0} & 0 & \begin{array}{ccc}
0 & 0 & 0\end{array}\\
0 & \rho_{2} & \begin{array}{ccc}
0 & 0 & 0\end{array}\\
\begin{array}{c}
0\\
0\\
0
\end{array} & \begin{array}{c}
0\\
0\\
0
\end{array} & W
\end{array}\right),\quad M_{1}\coloneqq\left(\begin{array}{ccccc}
0 & 0 & 0 & 0 & 0\\
0 & 0 & -\mathrm{skew} & 0 & -\iota_{\mathrm{sym}}\\
0 & \mathrm{skew} & 0 & 0 & 0\\
0 & 0 & 0 & 0 & 0\\
0 & \iota_{\mathrm{sym}}^{*} & 0 & 0 & 0
\end{array}\right)
\]
and
\begin{align*}
A & \coloneqq\left(\begin{array}{ccccc}
0 & 0 & -\nabla\cdot & 0 & 0\\
0 & 0 & 0 & -\nabla\cdot & 0\\
-\interior\nabla & 0 & 0 & 0 & 0\\
0 & -\interior\nabla & 0 & 0 & 0\\
0 & 0 & 0 & 0 & 0
\end{array}\right).
\end{align*}
\end{thm}
\begin{proof}
Before we show the assertion observe that for given 
\[
\left(\begin{array}{c}
\tau\\
\sigma\\
\mu
\end{array}\right)\in\mathrm{sym}\left[L^{2,2}\left(\Omega\right)\right]\oplus L^{2,2}\left(\Omega\right)\oplus L^{2,3}\left(\Omega\right),
\]
we get that
\[
\left(\begin{array}{c}
\iota_{\mathrm{sym}}\tau+\sigma\\
\mu\\
\iota_{\mathrm{sym}}^{*}\sigma
\end{array}\right)=\left(\begin{array}{ccc}
\iota_{\mathrm{sym}} & 1 & 0\\
0 & 0 & 1\\
0 & \iota_{\mathrm{sym}}^{*} & 0
\end{array}\right)\left(\begin{array}{c}
\tau\\
\sigma\\
\mu
\end{array}\right)\in L^{2,2}\left(\Omega\right)\oplus L^{2,3}\left(\Omega\right)\oplus\mathrm{sym}\left[L^{2,2}\left(\Omega\right)\right].
\]
Realizing that $\left(\begin{array}{ccc}
\iota_{\mathrm{sym}}^{*} & 0 & 0\\
1 & 0 & \iota_{\mathrm{sym}}\\
0 & 1 & 0
\end{array}\right)\left(\begin{array}{ccc}
\iota_{\mathrm{sym}} & \mathrm{skew} & 0\\
0 & 0 & 1\\
-1 & \iota_{\mathrm{sym}}^{*} & 0
\end{array}\right)=1$ and invoking $\left(\begin{array}{c}
\tau\\
\sigma\\
\mu
\end{array}\right)=\left(\begin{array}{ccc}
C_{0} & G_{0} & F_{0}\\
G_{0}^{*} & C_{1} & D_{0}\\
F_{0}^{*} & D_{0}^{*} & C_{2}
\end{array}\right)\left(\begin{array}{c}
\epsilon\\
\gamma\\
\kappa
\end{array}\right)$ for some $\epsilon,\gamma$ and $\kappa$, we arrive at
\begin{align*}
\left(\begin{array}{c}
\iota_{\mathrm{sym}}\tau+\sigma\\
\mu\\
\iota_{\mathrm{sym}}^{*}\sigma
\end{array}\right) & =\left(\begin{array}{ccc}
\iota_{\mathrm{sym}} & 1 & 0\\
0 & 0 & 1\\
0 & \iota_{\mathrm{sym}}^{*} & 0
\end{array}\right)\left(\begin{array}{c}
\tau\\
\sigma\\
\mu
\end{array}\right),\\
 & =\left(\left(\begin{array}{ccc}
\iota_{\mathrm{sym}} & 1 & 0\\
0 & 0 & 1\\
0 & \iota_{\mathrm{sym}}^{*} & 0
\end{array}\right)\left(\begin{array}{ccc}
C_{0} & G_{0} & F_{0}\\
G_{0}^{*} & C_{1} & D_{0}\\
F_{0}^{*} & D_{0}^{*} & C_{2}
\end{array}\right)\left(\begin{array}{ccc}
\iota_{\mathrm{sym}}^{*} & 0 & 0\\
1 & 0 & \iota_{\mathrm{sym}}\\
0 & 1 & 0
\end{array}\right)\right)\left(\left(\begin{array}{ccc}
\iota_{\mathrm{sym}} & \mathrm{skew} & 0\\
0 & 0 & 1\\
-1 & \iota_{\mathrm{sym}}^{*} & 0
\end{array}\right)\left(\begin{array}{c}
\epsilon\\
\gamma\\
\kappa
\end{array}\right)\right).
\end{align*}
From
\begin{align*}
 & \left(\begin{array}{ccc}
\iota_{\mathrm{sym}} & 1 & 0\\
0 & 0 & 1\\
0 & \iota_{\mathrm{sym}}^{*} & 0
\end{array}\right)\left(\begin{array}{ccc}
C_{0} & G_{0} & F_{0}\\
G_{0}^{*} & C_{1} & D_{0}\\
F_{0}^{*} & D_{0}^{*} & C_{2}
\end{array}\right)\left(\begin{array}{ccc}
\iota_{\mathrm{sym}}^{*} & 0 & 0\\
1 & 0 & \iota_{\mathrm{sym}}\\
0 & 1 & 0
\end{array}\right)=\\
 & =\left(\begin{array}{ccc}
\iota_{\mathrm{sym}}C_{0}+G_{0}^{*} & \;\iota_{\mathrm{sym}}G_{0}+C_{1} & \;\iota_{\mathrm{sym}}F_{0}+D_{0}\\
F_{0}^{*} & D_{0}^{*} & C_{2}\\
\iota_{\mathrm{sym}}^{*}G_{0}^{*} & \iota_{\mathrm{sym}}^{*}C_{1} & \iota_{\mathrm{sym}}^{*}D_{0}
\end{array}\right)\left(\begin{array}{ccc}
\iota_{\mathrm{sym}}^{*} & 0 & 0\\
1 & 0 & \iota_{\mathrm{sym}}\\
0 & 1 & 0
\end{array}\right)\\
 & =\left(\begin{array}{ccc}
\iota_{\mathrm{sym}}C_{0}\iota_{\mathrm{sym}}^{*}+G_{0}^{*}\iota_{\mathrm{sym}}^{*}+\iota_{\mathrm{sym}}G_{0}+C_{1} & \quad\iota_{\mathrm{sym}}F_{0}+D_{0} & \quad\iota_{\mathrm{sym}}G_{0}\iota_{\mathrm{sym}}+C_{1}\iota_{\mathrm{sym}}\\
F_{0}^{*}\iota_{\mathrm{sym}}^{*}+D_{0}^{*} & C_{2} & D_{0}^{*}\iota_{\mathrm{sym}}\\
\iota_{\mathrm{sym}}^{*}G_{0}^{*}\iota_{\mathrm{sym}}^{*}+\iota_{\mathrm{sym}}^{*}C_{1} & \iota_{\mathrm{sym}}^{*}D_{0} & \iota_{\mathrm{sym}}^{*}C_{1}\iota_{\mathrm{sym}}
\end{array}\right),
\end{align*}
we get that 
\begin{align}
W\left(\begin{array}{c}
\iota_{\mathrm{sym}}\tau+\sigma\\
\mu\\
\iota_{\mathrm{sym}}^{*}\sigma
\end{array}\right) & =\left(\begin{array}{ccc}
\iota_{\mathrm{sym}} & \mathrm{skew} & 0\\
0 & 0 & 1\\
-1 & \iota_{\mathrm{sym}}^{*} & 0
\end{array}\right)\left(\begin{array}{c}
\epsilon\\
\gamma\\
\kappa
\end{array}\right)\nonumber \\
 & =\left(\begin{array}{c}
\iota_{\mathrm{sym}}\epsilon+\mathrm{skew}\gamma\\
\kappa\\
\iota_{\mathrm{sym}}^{*}\gamma-\epsilon
\end{array}\right).\label{eq:W}
\end{align}
Now, let $(\dot{u},\dot{\psi})\in H_{\rho,0}(\mathbb{R};L^{2,1}(\Omega)\oplus L^{2,2}(\Omega))$
be a solution of (\ref{eq:dyn_micro_med}) subject to (\ref{eq:pre-matter}).
Applying $\partial_{0}$ to equation (\ref{eq:W}), using the time-independence
of $W$ and the definition of $\epsilon,\gamma$ and $\kappa$, we
get that
\[
\partial_{0}W\left(\begin{array}{c}
\iota_{\mathrm{sym}}\tau+\sigma\\
\mu\\
\iota_{\mathrm{sym}}^{*}\sigma
\end{array}\right)+\left(\begin{array}{c}
\mathrm{skew}\dot{\psi}\\
0\\
\iota_{\mathrm{sym}}^{*}\dot{\psi}
\end{array}\right)=\left(\begin{array}{c}
\nabla\dot{u}\\
\nabla\dot{\psi}\\
0
\end{array}\right).
\]
The latter equation together with (\ref{eq:dyn_micro_med}) yields
a solution of $(\partial_{0}M_{0}+M_{1}+A)U=\left(\begin{array}{c}
f\\
h\\
0\\
0\\
0
\end{array}\right).$ On the other hand, a solution of the latter equation gives equation
(\ref{eq:pre-matter}) by integrating the last three rows of the system,
i.e., by multiplying the last three rows by $\partial_{0}^{-1}$.\end{proof}
\begin{cor}
Let $\rho\in\mathbb{R}_{>0}$. The equation (\ref{eq:dyn_micro_med})
subject to (\ref{eq:pre-matter}) is well-posed in the sense that
for any $(f,h)\in H_{\rho,0}(\mathbb{R};L^{2,1}(\Omega)\oplus L^{2,2}(\Omega))$
there exists uniquely determined $(\dot{u},\dot{\psi})\in H_{\rho,0}(\mathbb{R};L^{2,1}(\Omega)\oplus L^{2,2}(\Omega))$
depending continuously on $(f,h)$ if $\rho_{0},\rho_{2}$ and $W$
are strictly positive definite. Moreover, in this case the energy
balance equality
\begin{align*}
 & \frac{1}{2}\left\langle \left(\begin{array}{c}
\dot{u}\\
\dot{\psi}\\
\iota_{\mathrm{sym}}\tau+\sigma\\
\mu\\
\iota_{\mathrm{sym}}^{\ast}\sigma
\end{array}\right)\left|\left(\begin{array}{c}
\rho_{0}\dot{u}\\
\rho_{2}\dot{\psi}\\
\iota_{\mathrm{sym}}\epsilon+\mathrm{skew}\gamma\\
\kappa\\
\iota_{\mathrm{sym}}^{\ast}\gamma-\epsilon
\end{array}\right)\right.\right\rangle (b)\\
 & =\frac{1}{2}\left\langle \left(\begin{array}{c}
\dot{u}\\
\dot{\psi}\\
\iota_{\mathrm{sym}}\tau+\sigma\\
\mu\\
\iota_{\mathrm{sym}}^{\ast}\sigma
\end{array}\right)\left|\left(\begin{array}{c}
\rho_{0}\dot{u}\\
\rho_{2}\dot{\psi}\\
\iota_{\mathrm{sym}}\epsilon+\mathrm{skew}\gamma\\
\kappa\\
\iota_{\mathrm{sym}}^{\ast}\gamma-\epsilon
\end{array}\right)\right.\right\rangle (a)+\intop_{a}^{b}\left\langle \left.\left(\begin{array}{c}
\dot{u}\\
\dot{\psi}
\end{array}\right)\right|\left(\begin{array}{c}
f\\
h
\end{array}\right)\right\rangle 
\end{align*}

holds for almost every $a,b\in\mathbb{R}$.\end{cor}
\begin{proof}
With the help of the reformulation done in Theorem \ref{thm:micro-med-reform},
the proof rests on Theorem \ref{linear-evolutionary-solution-theory}
for the solution theory and Theorem \ref{thm:energy balance} for
the energy balance.\end{proof}
\begin{rem}
We elaborate the condition of positive definiteness of $W$ more closely:
First we note that $W$ is strictly positive definite if and only
if $W^{-1}$ is strictly positive definite.\\
If in particular
\[
D_{0}=0,\: F_{0}=0
\]
we merely have to consider the positive definiteness of
\[
\left(\begin{array}{ccc}
\iota_{\mathrm{sym}}C_{0}\iota_{\mathrm{sym}}^{*}+G_{0}^{*}\iota_{\mathrm{sym}}^{*}+\iota_{\mathrm{sym}}G_{0}+C_{1} & \quad\iota_{\mathrm{sym}}G_{0}\iota_{\mathrm{sym}}+C_{1}\iota_{\mathrm{sym}} & 0\\
\iota_{\mathrm{sym}}^{*}G_{0}^{*}\iota_{\mathrm{sym}}^{*}+\iota_{\mathrm{sym}}^{*}C_{1} & \iota_{\mathrm{sym}}^{*}C_{1}\iota_{\mathrm{sym}} & 0\\
0 & 0 & C_{2}
\end{array}\right),
\]
which holds if $C_{2}$ is strictly positive definite and
\[
\left(\begin{array}{cc}
\iota_{\mathrm{sym}}C_{0}\iota_{\mathrm{sym}}^{*}+G_{0}^{*}\iota_{\mathrm{sym}}^{*}+\iota_{\mathrm{sym}}G_{0}+C_{1} & \quad\iota_{\mathrm{sym}}G_{0}\iota_{\mathrm{sym}}+C_{1}\iota_{\mathrm{sym}}\\
\iota_{\mathrm{sym}}^{*}G_{0}^{*}\iota_{\mathrm{sym}}^{*}+\iota_{\mathrm{sym}}^{*}C_{1} & \iota_{\mathrm{sym}}^{*}C_{1}\iota_{\mathrm{sym}}
\end{array}\right)
\]
is strictly positive definite. The latter holds if $\iota_{\mathrm{sym}}^{*}C_{1}\iota_{\mathrm{sym}}$
and 
\begin{align}
 & \iota_{\mathrm{sym}}C_{0}\iota_{\mathrm{sym}}^{*}+G_{0}^{*}\iota_{\mathrm{sym}}^{*}+\iota_{\mathrm{sym}}G_{0}+C_{1}\label{eq:upleft}\\
 & -\left(\iota_{\mathrm{sym}}G_{0}\iota_{\mathrm{sym}}+C_{1}\iota_{\mathrm{sym}}\right)\left(\iota_{\mathrm{sym}}^{*}C_{1}\iota_{\mathrm{sym}}\right)^{-1}\left(\iota_{\mathrm{sym}}^{*}G_{0}^{*}\iota_{\mathrm{sym}}^{*}+\iota_{\mathrm{sym}}^{*}C_{1}\right)\nonumber 
\end{align}
is strictly positive definite. 

In the isotropic case we do have $D_{0}=0$ and $F_{0}=0$ and moreover
\begin{align*}
C_{0}= & \iota_{\mathrm{sym}}^{*}\left(2\mu_{0}\mathrm{sym}_{0}+\left(3\lambda_{0}+2\mu_{0}\right)\mathbb{P}\right)\iota_{\mathrm{sym}},\\
C_{1}= & 2\mu_{1}\mathrm{sym}_{0}+2\alpha_{1}\mathrm{skew}+\left(3\lambda_{1}+2\mu_{1}\right)\mathbb{P},\\
G_{0}= & \iota_{\mathrm{sym}}^{*}\left(2\omega_{0}\mathrm{sym}_{0}+\left(3\beta_{0}+2\omega_{0}\right)\mathbb{P}\right).
\end{align*}
for suitable scalars $\mu_{0},\lambda_{0},\beta_{0},\omega_{0},\mu_{1},\lambda_{1},\alpha_{1}\in\mathbb{R}.$
With this we find that
\begin{align*}
\iota_{\mathrm{sym}}G_{0} & =2\omega_{0}\mathrm{sym}_{0}+\left(3\beta_{0}+2\omega_{0}\right)\mathbb{P}=G_{0}^{*}\iota_{\mathrm{sym}}^{*},\\
\iota_{\mathrm{sym}}G_{0}\iota_{\mathrm{sym}} & =\left(2\omega_{0}\mathrm{sym}_{0}+\left(3\beta_{0}+2\omega_{0}\right)\mathbb{P}\right)\iota_{\mathrm{sym}}
\end{align*}
 and so
\begin{align*}
 & \iota_{\mathrm{sym}}C_{0}\iota_{\mathrm{sym}}^{*}+G_{0}^{*}\iota_{\mathrm{sym}}^{*}+\iota_{\mathrm{sym}}G_{0}+C_{1}\\
 & =2\left(\mu_{0}+\mu_{1}+2\omega_{0}\right)\mathrm{sym}_{0}+2\alpha_{1}\mathrm{skew}+\left(3\left(2\beta_{0}+\lambda_{0}+\lambda_{1}\right)+2\left(2\omega_{0}+\mu_{0}+\mu_{1}\right)\right)\mathbb{P}.
\end{align*}
Moreover, with 
\[
\mu_{1}>0,\:\lambda_{1}+\frac{2}{3}\mu_{1}>0,
\]
we find
\begin{align*}
\left(\iota_{\mathrm{sym}}^{*}C_{1}\iota_{\mathrm{sym}}\right)^{-1}= & \iota_{\mathrm{sym}}^{*}\left(\frac{1}{2\mu_{1}}\mathrm{sym}_{0}+\frac{1}{3\lambda_{1}+2\mu_{1}}\mathbb{P}\right)\iota_{\mathrm{sym}},\\
\iota_{\mathrm{sym}}G_{0}\iota_{\mathrm{sym}}+C_{1}\iota_{\mathrm{sym}}= & \left(2\left(\omega_{0}+\mu_{1}\right)\mathrm{sym}_{0}+\left(3\left(\beta_{0}+\lambda_{1}\right)+2\left(\omega_{0}+\mu_{1}\right)\right)\mathbb{P}\right)\iota_{\mathrm{sym}}
\end{align*}
and so
\begin{align*}
 & \left(\iota_{\mathrm{sym}}G_{0}\iota_{\mathrm{sym}}+C_{1}\iota_{\mathrm{sym}}\right)\left(\iota_{\mathrm{sym}}^{*}C_{1}\iota_{\mathrm{sym}}\right)^{-1}\left(\iota_{\mathrm{sym}}^{*}G_{0}^{*}\iota_{\mathrm{sym}}^{*}+\iota_{\mathrm{sym}}^{*}C_{1}\right)\\
 & =2\frac{\left(\omega_{0}+\mu_{1}\right)^{2}}{\mu_{1}}\mathrm{sym}_{0}+\frac{\left(3\left(\beta_{0}+\lambda_{1}\right)+2\left(\omega_{0}+\mu_{1}\right)\right)^{2}}{3\lambda_{1}+2\mu_{1}}\mathbb{P}.
\end{align*}
The positive definiteness of (\ref{eq:upleft}) now follows if 
\begin{align*}
 & 2\left(\mu_{0}+\mu_{1}+2\omega_{0}-\frac{\left(\omega_{0}+\mu_{1}\right)^{2}}{\mu_{1}}\right)\mathrm{sym}_{0}+2\alpha_{1}\mathrm{skew}\\
 & +\left(3\left(2\beta_{0}+\lambda_{0}+\lambda_{1}\right)+2\left(2\omega_{0}+\mu_{0}+\mu_{1}\right)-\frac{\left(3\left(\beta_{0}+\lambda_{1}\right)+2\left(\omega_{0}+\mu_{1}\right)\right)^{2}}{3\lambda_{1}+2\mu_{1}}\right)\mathbb{P}
\end{align*}
is positive, which is the case if
\end{rem}
\[
\alpha_{1}>0,
\]
\[
\mu_{1}\left(\mu_{0}+\mu_{1}+2\omega_{0}\right)-\left(\omega_{0}+\mu_{1}\right)^{2}=\mu_{1}\mu_{0}-\omega_{0}^{2}>0
\]
and
\begin{align*}
 & \left(3\lambda_{1}+2\mu_{1}\right)\left(3\left(2\beta_{0}+\lambda_{0}+\lambda_{1}\right)+2\left(2\omega_{0}+\mu_{0}+\mu_{1}\right)\right)-\left(3\left(\beta_{0}+\lambda_{1}\right)+2\left(\omega_{0}+\mu_{1}\right)\right)^{2}=\\
 & =\left(3\lambda_{1}+2\mu_{1}\right)^{2}+\left(3\lambda_{1}+2\mu_{1}\right)\left(2\left(3\beta_{0}+2\omega_{0}\right)+\left(3\lambda_{0}+2\mu_{0}\right)\right)-\left(\left(3\lambda_{1}+2\mu_{1}\right)+\left(3\beta_{0}+2\omega_{0}\right)\right)^{2}\\
 & =\left(3\lambda_{1}+2\mu_{1}\right)\left(3\lambda_{0}+2\mu_{0}\right)-\left(3\beta_{0}+2\omega_{0}\right)^{2}>0.
\end{align*}

\subsection{\label{sub:Micropolar-or-Cosserat}Micropolar or Cosserat Media}

The equations for Cosserat elasticity in the case of Dirichlet boundary
conditions read as
\begin{equation}
\left(\partial_{0}\left(\begin{array}{cccc}
\varrho_{0} & 0 & 0 & 0\\
0 & \varrho_{1} & 0 & 0\\
0 & 0 & C_{0}^{-1} & 0\\
0 & 0 & 0 & C_{1}^{-1}
\end{array}\right)+\left(\begin{array}{cccc}
0 & \quad0 & \quad0 & \quad0\\
0 & \quad0 & \quad-\Lambda^{*} & \quad0\\
0 & \quad\Lambda & \quad0 & \quad0\\
0 & \quad0 & \quad0 & \quad0
\end{array}\right)+\left(\begin{array}{cccc}
0 & 0 & -\nabla\cdot & 0\\
0 & 0 & 0 & -\nabla\cdot\\
-\interior\nabla & 0 & 0 & 0\\
0 & -\interior\nabla & 0 & 0
\end{array}\right)\right)\left(\begin{array}{c}
\mathbf{v}\\
\mathbf{w}\\
\sigma\\
\tau
\end{array}\right)=\left(\begin{array}{c}
\mathbf{f}\\
\mathbf{g}\\
0\\
0
\end{array}\right),\label{eq:Cosserat}
\end{equation}
where $\Lambda\coloneqq2\iota_{\wedge}*$, $\Lambda^{*}\coloneqq2*\iota_{\wedge}^{*}$
with $*$ denoting the Hodge star operator and $\iota_{\wedge}$ the
canonical embedding of alternating differential forms into the space
of $2$-tensors, i.e. 
\[
\iota_{\wedge}\left(dx^{i}\wedge dx^{j}\right)=\frac{1}{2}\left(dx^{i}\otimes dx^{j}-dx^{j}\otimes dx^{i}\right)
\]
for $i,j\in\{1,2,3\}.$ As underlying Hilbert space we have

\[
H=L^{2,1}\left(\Omega\right)\oplus L^{2,1}\left(\Omega\right)\oplus L^{2,2}\left(\Omega\right)\oplus L^{2,2}\left(\Omega\right).
\]

In Cartesian coordinates using an obvious suggestive matrix notation
we have
\begin{align*}
\Lambda^{*}\left(\left(\begin{array}{ccc}
\alpha_{11} & \alpha_{12} & \alpha_{13}\\
\alpha_{21} & \alpha_{22} & \alpha_{23}\\
\alpha_{31} & \alpha_{32} & \alpha_{33}
\end{array}\right)\right) & =\left(\begin{array}{c}
\alpha_{23}-\alpha_{32}\\
\alpha_{31}-\alpha_{13}\\
\alpha_{12}-\alpha_{21}
\end{array}\right)
\end{align*}
and
\[
-\Lambda\left(\left(\begin{array}{c}
\beta_{1}\\
\beta_{2}\\
\beta_{3}
\end{array}\right)\right)=\left(\begin{array}{ccc}
0 & -\beta_{3} & \beta_{2}\\
\beta_{3} & 0 & -\beta_{1}\\
-\beta_{2} & \beta_{1} & 0
\end{array}\right)\eqqcolon\beta\times
\]
for suitable $\alpha\in L^{2,2}(\Omega),\beta\in L^{2,1}(\Omega)$.
\begin{rem}
\label{rem:lambda_star_skew_unitary}On skew-symmetric tensors we
have that
\[
\Lambda^{*}\left(\left(\begin{array}{ccc}
0 & \alpha_{12} & -\alpha_{31}\\
-\alpha_{12} & 0 & \alpha_{23}\\
\alpha_{31} & \alpha_{-23} & 0
\end{array}\right)\right)=2\left(\begin{array}{c}
\alpha_{23}\\
\alpha_{31}\\
\alpha_{12}
\end{array}\right).
\]
Thus, we have that $\frac{1}{\sqrt{2}}\Lambda^{*}\iota_{\mathrm{skew}}:\iota_{\mathrm{skew}}^{*}\left[L^{2,2}\left(\Omega\right)\right]\to L^{2,1}\left(\Omega\right)$
defines a unitary transformation.
\end{rem}
Having stated the model for Cosserat elasticity (\ref{eq:Cosserat})
in the canonical form of Theorem \ref{linear-evolutionary-solution-theory},
we are in the position of formulating the well-posedness result as
follows.
\begin{thm}
\label{thm:WP_Cosserat}Let $\rho\in\mathbb{R}_{>0}$ and let $\rho_{0},\rho_{1},C_{0},C_{1}$
be selfadjoint and strictly positive definite operators in the respective
Hilbert spaces $L^{2,1}(\text{\ensuremath{\Omega})},L^{2,1}(\Omega),L^{2,2}(\Omega)$
and $L^{2,2}(\Omega)$. Then for every $(\mathbf{f},\mathbf{g})\in H_{\rho,0}(\mathbb{R};L^{2,1}(\Omega)\oplus L^{2,1}(\Omega))$
there exists uniquely determined $\mbox{(\ensuremath{\mathbf{v}},\ensuremath{\mathbf{w}},\ensuremath{\sigma},\ensuremath{\tau}})\in H_{\rho,0}(\mathbb{R};L^{2,1}(\Omega)\oplus L^{2,1}(\Omega)\oplus L^{2,2}(\Omega)\oplus L^{2,2}(\Omega))$
satisfying (\ref{eq:Cosserat}) depending continuously on $(\mathbf{f},\mathbf{g})$.
Moreover, the energy balance equality 
\[
\frac{1}{2}\left\langle \left(\begin{array}{c}
\mathbf{v}\\
\mathbf{w}\\
\sigma\\
\tau
\end{array}\right)\left|\left(\begin{array}{c}
\rho_{0}\mathbf{v}\\
\rho_{1}\mathbf{w}\\
C_{0}^{-1}\sigma\\
C_{1}^{-1}\tau
\end{array}\right)\right.\right\rangle \rangle(b)=\frac{1}{2}\left\langle \left(\begin{array}{c}
\mathbf{v}\\
\mathbf{w}\\
\sigma\\
\tau
\end{array}\right)\left|\left(\begin{array}{c}
\rho_{0}\mathbf{v}\\
\rho_{1}\mathbf{w}\\
C_{0}^{-1}\sigma\\
C_{1}^{-1}\tau
\end{array}\right)\right.\right\rangle (a)+\int_{a}^{b}\Re\left\langle \left(\begin{array}{c}
\mathbf{v}\\
\mathbf{w}
\end{array}\right)\left|\left(\begin{array}{c}
\mathbf{f}\\
\mathbf{g}
\end{array}\right)\right.\right\rangle 
\]
holds for almost every $a,b\in\mathbb{R}$.\end{thm}
\begin{proof}
The assertion follows easily from the Theorems \ref{linear-evolutionary-solution-theory}
and \ref{thm:energy balance}.\end{proof}
\begin{rem}
We briefly discuss the relationship of (\ref{eq:Cosserat}) to a model
discussed in \cite{zbMATH03490546}. From (\ref{eq:Cosserat}) with
$\partial_{0}\phi\coloneqq\mathbf{w}$ and $\partial_{0}u\coloneqq\mathbf{v}$
we read off
\begin{align*}
\varrho_{0}\partial_{0}^{2}u-\nabla\cdot\sigma & =\mathbf{f},\\
\varrho_{1}\partial_{0}^{2}\phi-\Lambda^{*}\sigma-\nabla\cdot\tau & =\mathbf{g},\\
\partial_{0}\sigma & =C_{0}\left(\nabla\partial_{0}u-\partial_{0}\Lambda\phi\right),\\
\partial_{0}\tau & =C_{1}\nabla\partial_{0}\phi.
\end{align*}
Applying $\partial_{0}^{-1}$ to the third and the fourth equation
we get in summary
\begin{align*}
\varrho_{0}\partial_{0}^{2}u-\nabla\cdot\sigma & =\mathbf{f},\\
\varrho_{1}\partial_{0}^{2}\phi-\Lambda^{*}\sigma-\nabla\cdot\tau & =\mathbf{g},\\
\sigma & =C_{0}\left(\nabla u+\phi\times\right),\\
\tau & =C_{1}\nabla\phi.
\end{align*}
Comparing with \cite[p.1--43]{zbMATH03490546} we see that this coincides
with the equations formally given there in the isotropic case, where
\[
C_{0}=2\alpha_{0}\mathrm{skew}+2\mu_{0}\mathrm{sym}+3\lambda_{0}\mathbb{P},
\]
\[
C_{1}=2\alpha_{1}\mathrm{skew}+2\mu_{1}\mathrm{sym}+3\lambda_{1}\mathbb{P}.
\]
for $\alpha_{0},\mu_{0},\lambda_{0},\alpha_{1},\mu_{1},\lambda_{1}\in\mathbb{R}.$
The needed positive definiteness of $C_{0},\: C_{1}$ can be conveniently
analyzed by the mechanism of Remark \ref{rem-block} which yields
the unitary equivalence
\begin{equation}
\left(\begin{array}{c}
\iota_{\mathrm{sym}_{0}}^{*}\\
\iota_{\mathrm{skew}}^{*}\\
\iota_{\mathbb{P}}^{*}
\end{array}\right)C_{k}\left(\begin{array}{ccc}
\iota_{\mathrm{sym}_{0}} & \iota_{\mathrm{skew}} & \iota_{\mathbb{P}}\end{array}\right)=\left(\begin{array}{ccc}
\iota_{\mathrm{sym}_{0}}^{*}C_{k}\iota_{\mathrm{sym}_{0}} & \iota_{\mathrm{sym}_{0}}^{*}C_{k}\iota_{\mathrm{skew}} & \iota_{\mathrm{sym}_{0}}^{*}C_{k}\iota_{\mathbb{P}}\\
\iota_{\mathrm{skew}}^{*}C_{k}\iota_{\mathrm{sym}_{0}} & \iota_{\mathrm{skew}}^{*}C_{k}\iota_{\mathrm{skew}} & \iota_{\mathrm{skew}}^{*}C_{k}\iota_{\mathbb{P}}\\
\iota_{\mathbb{P}}^{*}C_{k}\iota_{\mathrm{sym}_{0}} & \iota_{\mathbb{P}}^{*}C_{k}\iota_{\mathrm{skew}} & \iota_{\mathbb{P}}^{*}C_{k}\iota_{\mathbb{P}}
\end{array}\right),\: k\in\{0,1\}.\label{eq:non-Voigt}
\end{equation}
In the isotropic case we thus see that positive definiteness is ensured
if (and only if)
\[
\left(\begin{array}{ccc}
2\mu_{0} & 0 & 0\\
0 & 2\alpha_{0} & 0\\
0 & 0 & 3\left(\lambda_{0}+\frac{2}{3}\mu_{0}\right)
\end{array}\right),\left(\begin{array}{ccc}
2\mu_{1} & 0 & 0\\
0 & 2\alpha_{1} & 0\\
0 & 0 & 3\left(\lambda_{1}+\frac{2}{3}\mu_{1}\right)
\end{array}\right)
\]
are positive definite. In other words,
\[
\mu_{k},\alpha_{k},\lambda_{k}+\frac{2}{3}\mu_{k}>0,\: k\in\{0,1\}.
\]
It appears that the block matrix conversion (\ref{eq:non-Voigt})
is easier to use for obtaining positive definiteness conditions than
the usually employed Voigt type notation. The latter results from
the same mechanism by using the canonical projectors onto the relevant
Euclidean components rather than the three projectors used in (\ref{eq:non-Voigt}). 
\end{rem}
In the remainder of this section, we show that the system for micromorphic
media (see Theorem \ref{thm:micro-med-reform}) is a mother of a relative
of the equations of Cosserat elasticity (\ref{eq:Cosserat}). For
this we discuss a relative of Cosserat elasticity first:
\begin{thm}
\label{thm:relative_Cosserat}The following equations are the $\left(\begin{array}{cccc}
1 & 0 & 0 & 0\\
0 & \iota_{\mathrm{skew}}^{*}\frac{1}{\sqrt{2}}\Lambda & 0 & 0\\
0 & 0 & 1 & 0\\
0 & 0 & 0 & 1\otimes\iota_{\mathrm{skew}}^{*}\frac{1}{\sqrt{2}}\Lambda
\end{array}\right)$-relative%
\footnote{Here $1\otimes\iota_{\mathrm{skew}}^{*}\Lambda^{*}$ indicates that
$\iota_{\mathrm{skew}}^{*}\Lambda^{*}$ is only to be applied with
respect to the last two arguments. More precisely, if $F$ maps 2-linear
forms to 2-linear forms then $1\otimes F$ maps 3-tensors to 3-tensors
\begin{equation}
\left(\left(1\otimes F\right)T\right)\left(x,y,z\right)=\left(FT\left(x,\:\cdot\:,\:\cdot\:\right)\right)\left(y,z\right)\label{eq:tensor-product-map}
\end{equation}
and we use the notation $1\otimes F$ also for the canonical extension
to $L^{2,3}\left(\Omega\right)$, when $F$ can be canonically extended
to $L^{2,2}\left(\Omega\right)$, currently for $F=\iota_{\mathrm{skew}}^{*}\Lambda$
or $F=\Lambda^{*}\iota_{\mathrm{skew}}$.%
} of the model of Cosserat elasticity (\ref{eq:Cosserat}): 
\begin{align}
 & \left(\partial_{0}M_{0}+M_{1}+A\right)\left(\begin{array}{c}
\mathbf{v}\\
\omega\\
\sigma\\
\mu
\end{array}\right)=\left(\begin{array}{c}
\mathbf{f}\\
\mathbf{h}\\
0\\
0
\end{array}\right)\label{eq:rel_of_cosserat}\\
 & \in H_{\rho,0}(\mathbb{R};L^{2,1}\left(\Omega\right)\oplus\mathrm{skew}\left[L^{2,2}\left(\Omega\right)\right]\oplus L^{2,2}\left(\Omega\right)\oplus\left(1\otimes\mathrm{skew}\right)\left[L^{2,3}\left(\Omega\right)\right])\nonumber 
\end{align}
with
\[
M_{0}\coloneqq\left(\begin{array}{cccc}
\varrho_{0} & 0 & 0 & 0\\
0 & \varrho_{2} & 0 & 0\\
0 & 0 & C_{0}^{-1} & 0\\
0 & 0 & 0 & C_{2}^{-1}
\end{array}\right),
\]
\[
M_{1}\coloneqq\left(\begin{array}{cccc}
0 & 0 & 0 & 0\\
0 & 0 & -\sqrt{2}\iota_{\mathrm{skew}}^{*} & 0\\
0 & \sqrt{2}\iota_{\mathrm{skew}} & 0 & 0\\
0 & 0 & 0 & 0
\end{array}\right)
\]
and
\[
A\coloneqq\left(\begin{array}{cccc}
0 & 0 & -\nabla\cdot & 0\\
0 & 0 & 0 & -\overline{\iota_{\mathrm{skew}}^{*}\nabla\cdot}\left(1\otimes\iota_{\mathrm{skew}}\right)\\
-\interior\nabla & 0 & 0 & 0\\
0 & -\overline{\left(1\otimes\iota_{\mathrm{skew}}^{*}\right)\interior\nabla}\iota_{\mathrm{skew}} & 0 & 0
\end{array}\right),
\]
 where 
\[
\varrho_{2}=\frac{1}{2}\iota_{\mathrm{skew}}^{*}\Lambda\varrho_{1}\Lambda^{*}\iota_{\mathrm{skew}},
\]
\[
C_{2}=\frac{1}{2}\left(1\otimes\iota_{\mathrm{skew}}^{*}\Lambda\right)C_{1}\left(1\otimes\Lambda^{*}\iota_{\mathrm{skew}}\right).
\]
\end{thm}
\begin{proof}
Using Remark \ref{rem:lambda_star_skew_unitary}, we see that $\left(\begin{array}{cccc}
1 & 0 & 0 & 0\\
0 & \iota_{\mathrm{skew}}^{*}\frac{1}{\sqrt{2}}\Lambda & 0 & 0\\
0 & 0 & 1 & 0\\
0 & 0 & 0 & 1\otimes\iota_{\mathrm{skew}}^{*}\frac{1}{\sqrt{2}}\Lambda
\end{array}\right)$ is indeed a bijection. The compatibility conditions are easily verified.
The remaining parts are verified by direct computation.\end{proof}
\begin{rem}
\label{rem:An-alternative-perspective} $\,$

\begin{enumerate}[(a)]

\item Of course, the relative of the model for Cosserat elasticity
has the well-posedness requirements that the operators
\begin{align*}
\varrho_{0}:L^{2,1}\left(\Omega\right) & \to L^{2,1}\left(\Omega\right)\\
\varrho_{2}:\mathrm{skew}\left[L^{2,2}\left(\Omega\right)\right] & \to\mathrm{skew}\left[L^{2,2}\left(\Omega\right)\right]\\
C_{0}:L^{2,2}\left(\Omega\right) & \to L^{2,2}\left(\Omega\right)\\
C_{2}:\left(1\otimes\mathrm{skew}\right)\left[L^{2,3}\left(\Omega\right)\right] & \to\left(1\otimes\mathrm{skew}\right)\left[L^{2,3}\left(\Omega\right)\right]
\end{align*}
are strictly positive definite. In view of Remark \ref{rem:lambda_star_skew_unitary},
the latter requirements are equivalent to the well-posedness requirements
in Theorem \ref{thm:WP_Cosserat}.

\item We note that by a suitable scaling of the unknown $(\mathbf{v},\omega,\sigma,\mu)$
and of the source term $(\mathbf{f,h}),$ system (\ref{eq:rel_of_cosserat})
can be brought into the form 
\begin{equation}
\left(\partial_{0}\tilde{M_{0}}+\tilde{M}_{1}+A\right)\left(\begin{array}{c}
\tilde{\mathbf{v}}\\
\tilde{\omega}\\
\tilde{\sigma}\\
\tilde{\mu}
\end{array}\right)=\left(\begin{array}{c}
\tilde{\mathbf{f}}\\
\tilde{\mathbf{h}}\\
0\\
0
\end{array}\right),\label{eq:rescaled}
\end{equation}
where $\tilde{M}_{0}$ is unitarily equivalent to $M_{0}$ and $\tilde{M}_{1}$
is given by 
\[
\left(\begin{array}{cccc}
0 & 0 & 0 & 0\\
0 & 0 & -\iota_{\mathrm{skew}}^{*} & 0\\
0 & \iota_{\mathrm{skew}} & 0 & 0\\
0 & 0 & 0 & 0
\end{array}\right).
\]

\end{enumerate}
\end{rem}
Next, we show that the relative of Cosserat elasticity introduced
above is a descendant of the model of micromorphic media from Theorem
\ref{thm:micro-med-reform}. It has already been pointed out by R.D.~Mindlin
(\cite{0119.40302}) that the model for micromorphic media contains
Cosserat elasticity as a special case. The precise connection is as
follows:
\begin{thm}
The system (\ref{eq:rescaled}) is a $\left(\begin{array}{ccccc}
1 & 0 & 0 & 0 & 0\\
0 & \iota_{\mathrm{skew}}^{*} & 0 & 0 & 0\\
0 & 0 & 1 & 0 & 0\\
0 & 0 & 0 & 1\otimes\iota_{\mathrm{skew}}^{*} & 0\\
0 & 0 & 0 & 0 & \iota_{0}^{*}
\end{array}\right)$-descendant of the model discussed in Theorem \ref{thm:micro-med-reform},
where $\iota_{0}\colon\{0\}\to L^{2,3}(\Omega),\phi\mapsto0$.\end{thm}
\begin{proof}
The computations 
\begin{align*}
 & \left(\begin{array}{ccccc}
1 & 0 & 0 & 0 & 0\\
0 & \iota_{\mathrm{skew}}^{*} & 0 & 0 & 0\\
0 & 0 & 1 & 0 & 0\\
0 & 0 & 0 & 1\otimes\iota_{\mathrm{skew}}^{*} & 0\\
0 & 0 & 0 & 0 & \iota_{0}^{*}
\end{array}\right)\left(\begin{array}{ccccc}
0 & 0 & 0 & 0 & 0\\
0 & 0 & -\mathrm{skew} & 0 & -\iota_{\mathrm{sym}}\\
0 & \mathrm{skew} & 0 & 0 & 0\\
0 & 0 & 0 & 0 & 0\\
0 & \iota_{\mathrm{sym}}^{*} & 0 & 0 & 0
\end{array}\right)\left(\begin{array}{ccccc}
1 & 0 & 0 & 0 & 0\\
0 & \iota_{\mathrm{skew}} & 0 & 0 & 0\\
0 & 0 & 1 & 0 & 0\\
0 & 0 & 0 & 1\otimes\iota_{\mathrm{skew}} & 0\\
0 & 0 & 0 & 0 & \iota_{0}
\end{array}\right)=\\
 & =\left(\begin{array}{ccccc}
0 & 0 & 0 & 0 & 0\\
0 & 0 & -\iota_{\mathrm{skew}}^{*} & 0 & 0\\
0 & \iota_{\mathrm{skew}} & 0 & 0 & 0\\
0 & 0 & 0 & 0 & 0\\
0 & 0 & 0 & 0 & 0
\end{array}\right),
\end{align*}

\begin{align*}
 & \left(\begin{array}{ccccc}
1 & 0 & 0 & 0 & 0\\
0 & \iota_{\mathrm{skew}}^{*} & 0 & 0 & 0\\
0 & 0 & 1 & 0 & 0\\
0 & 0 & 0 & 1\otimes\iota_{\mathrm{skew}}^{*} & 0\\
0 & 0 & 0 & 0 & \iota_{0}^{*}
\end{array}\right)\left(\begin{array}{ccccc}
\varrho_{0} & 0 & 0 & 0 & 0\\
0 & \varrho_{2} & 0 & 0 & 0\\
0 & 0 & W_{00} & W_{01} & W_{02}\\
0 & 0 & W_{01}^{*} & W_{11} & W_{12}\\
0 & 0 & W_{02}^{*} & W_{12}^{*} & W_{22}
\end{array}\right)\left(\begin{array}{ccccc}
1 & 0 & 0 & 0 & 0\\
0 & \iota_{\mathrm{skew}} & 0 & 0 & 0\\
0 & 0 & 1 & 0 & 0\\
0 & 0 & 0 & 1\otimes\iota_{\mathrm{skew}} & 0\\
0 & 0 & 0 & 0 & \iota_{0}
\end{array}\right)\\
 & =\left(\begin{array}{ccccc}
\varrho_{0} & 0 & 0 & 0 & 0\\
0 & \iota_{\mathrm{skew}}^{*}\varrho_{2}\iota_{\mathrm{skew}} & 0 & 0 & 0\\
0 & 0 & W_{00} & W_{01}\left(1\otimes\iota_{\mathrm{skew}}\right) & 0\\
0 & 0 & \left(1\otimes\iota_{\mathrm{skew}}^{*}\right)W_{01}^{*} & \left(1\otimes\iota_{\mathrm{skew}}^{*}\right)W_{11}\left(1\otimes\iota_{\mathrm{skew}}\right) & 0\\
0 & 0 & 0 & 0 & 0
\end{array}\right)
\end{align*}
and
\begin{align*}
 & \overline{\left(\begin{array}{ccccc}
1 & 0 & 0 & 0 & 0\\
0 & \iota_{\mathrm{skew}}^{*} & 0 & 0 & 0\\
0 & 0 & 1 & 0 & 0\\
0 & 0 & 0 & 1\otimes\iota_{\mathrm{skew}}^{*} & 0\\
0 & 0 & 0 & 0 & \iota_{0}^{*}
\end{array}\right)\left(\begin{array}{ccccc}
0 & 0 & -\nabla\cdot & 0 & 0\\
0 & 0 & 0 & -\nabla\cdot & 0\\
-\interior\nabla & 0 & 0 & 0 & 0\\
0 & -\interior\nabla & 0 & 0 & 0\\
0 & 0 & 0 & 0 & 0
\end{array}\right)}\left(\begin{array}{ccccc}
1 & 0 & 0 & 0 & 0\\
0 & \iota_{\mathrm{skew}} & 0 & 0 & 0\\
0 & 0 & 1 & 0 & 0\\
0 & 0 & 0 & 1\otimes\iota_{\mathrm{skew}} & 0\\
0 & 0 & 0 & 0 & \iota_{0}
\end{array}\right)=\\
 & =\left(\begin{array}{ccccc}
0 & 0 & -\nabla\cdot & 0 & 0\\
0 & 0 & 0 & -\overline{\iota_{\mathrm{skew}}^{*}\nabla\cdot}\left(1\otimes\iota_{\mathrm{skew}}\right) & 0\\
-\interior\nabla & 0 & 0 & 0 & 0\\
0 & -\overline{\left(1\otimes\iota_{\mathrm{skew}}^{*}\right)\interior\nabla}\iota_{\mathrm{skew}} & 0 & 0 & 0\\
0 & 0 & 0 & 0 & 0
\end{array}\right)
\end{align*}
show the assertion.\end{proof}
\begin{rem}[Hemitropic Media]
Following the generalizations of the theory of Cosserat media in
\cite{MR0162404} we are led to consider instead of
\[
\left(\begin{array}{c}
\sigma\\
\tau
\end{array}\right)=\left(\begin{array}{cc}
C_{0} & 0\\
0 & C_{1}
\end{array}\right)\left(\left(\begin{array}{c}
\nabla u\\
\nabla\phi
\end{array}\right)+\left(\begin{array}{c}
\phi\times\\
0
\end{array}\right)\right)
\]
as in (\ref{eq:Cosserat}) the more general material constraint
\[
\left(\begin{array}{c}
\sigma\\
\tau
\end{array}\right)=\left(\begin{array}{cc}
C_{0} & E^{*}\\
E & C_{2}
\end{array}\right)\left(\left(\begin{array}{c}
\nabla u\\
\nabla\phi
\end{array}\right)+\left(\begin{array}{c}
\phi\times\\
0
\end{array}\right)\right)
\]
or more symmetrically, compare \cite{1136.74018}, 
\[
\left(\begin{array}{c}
\sigma\\
\tau
\end{array}\right)=\left(\begin{array}{cc}
C_{0} & E^{*}\\
E & C_{2}
\end{array}\right)\left(\left(\begin{array}{c}
\nabla u\\
\nabla\phi
\end{array}\right)+\left(\begin{array}{c}
\phi\times\\
u\times
\end{array}\right)\right).
\]
By a symmetric Gauss elimination we first see that strict positive
definiteness of $\left(\begin{array}{cc}
C_{0} & E^{*}\\
E & C_{2}
\end{array}\right)$ in $L^{2,2}\left(\Omega\right)\oplus L^{2,2}\left(\Omega\right)$
is equivalent to the strict positive definiteness of 
\begin{equation}
\left(\begin{array}{cc}
1 & 0\\
-EC_{0}^{-1} & 1
\end{array}\right)\left(\begin{array}{cc}
C_{0} & E^{*}\\
E & C_{2}
\end{array}\right)\left(\begin{array}{cc}
1 & -C_{0}^{-1}E^{*}\\
0 & 1
\end{array}\right)=\left(\begin{array}{cc}
C_{0} & 0\\
0 & C_{2}-EC_{0}^{-1}E^{*}
\end{array}\right).\label{eq:symGauss0}
\end{equation}
Following the construction in (\ref{eq:non-Voigt}) we see therefore
that in the isotropic case, where we would have
\[
C_{0}=2\mu_{0}\mathrm{sym}+2\alpha_{0}\mathrm{skew}+3\lambda_{0}\mathbb{P}=2\mu_{0}\mathrm{sym}_{0}+2\alpha_{0}\mathrm{skew}+\left(3\lambda_{0}+2\mu_{0}\right)\mathbb{P}
\]
as before and 
\begin{align*}
C_{2} & =2\mu_{2}\mathrm{sym}+2\alpha_{2}\mathrm{skew}+3\lambda_{2}\mathbb{P}=2\mu_{2}\mathrm{sym}_{0}+2\alpha_{2}\mathrm{skew}+\left(3\lambda_{2}+2\mu_{2}\right)\mathbb{P},\\
E & =2\kappa_{0}\mathrm{sym}+2\nu_{0}\mathrm{skew}+3\delta_{0}\mathbb{P}=2\kappa_{0}\mathrm{sym}_{0}+2\nu_{0}\mathrm{skew}+\left(3\delta_{0}+2\kappa_{0}\right)\mathbb{P},
\end{align*}
for $\mu_{0},\alpha_{0},\lambda_{0},\mu_{2},\alpha_{2},\lambda_{2},\kappa_{0},\nu_{0},\delta_{0}\in\mathbb{R}$,
positive definiteness can be ensured if (and only if) 
\[
C_{0}=\left(\begin{array}{ccc}
2\mu_{0} & 0 & 0\\
0 & 2\alpha_{0} & 0\\
0 & 0 & 3\left(\lambda_{0}+\frac{2}{3}\mu_{0}\right)
\end{array}\right)
\]
and 
\begin{align*}
 & C_{2}-EC_{0}^{-1}E^{*}=\\
 & =\left(\begin{array}{ccc}
2\frac{\mu_{0}\mu_{2}-\kappa_{0}^{2}}{\mu_{0}} & 0 & 0\\
0 & 2\frac{\alpha_{0}\alpha_{2}-\nu_{0}^{2}}{\alpha_{0}} & 0\\
0 & 0 & 3\frac{\left(\lambda_{2}+\frac{2}{3}\mu_{2}\right)\left(\lambda_{0}+\frac{2}{3}\mu_{0}\right)-\left(\delta_{0}+\frac{2}{3}\kappa_{0}\right)\left(\delta_{0}+\frac{2}{3}\kappa_{0}\right)}{\left(\lambda_{0}+\frac{2}{3}\mu_{0}\right)}
\end{array}\right)
\end{align*}
are both positive definite. In other words, we must have
\[
\mu_{0},\alpha_{0},\lambda_{0}+\frac{2}{3}\mu_{0},\,\mu_{0}\mu_{2}-\kappa_{0}^{2},\alpha_{0}\alpha_{2}-\nu_{0}^{2},\,\left(\lambda_{2}+\frac{2}{3}\mu_{2}\right)\left(\lambda_{0}+\frac{2}{3}\mu_{0}\right)-\left(\delta_{0}+\frac{2}{3}\kappa_{0}\right)\left(\delta_{0}+\frac{2}{3}\kappa_{0}\right)>0.
\]
For finding $M_{0}$ we need to invert $\left(\begin{array}{cc}
C_{0} & E^{*}\\
E & C_{2}
\end{array}\right)$. It is with (\ref{eq:symGauss0})
\begin{align*}
\left(\begin{array}{cc}
1 & -C_{0}^{-1}E^{*}\\
0 & 1
\end{array}\right)^{-1}\left(\begin{array}{cc}
C_{0} & E^{*}\\
E & C_{2}
\end{array}\right)^{-1}\left(\begin{array}{cc}
1 & 0\\
-EC_{0}^{-1} & 1
\end{array}\right)^{-1} & =\left(\begin{array}{cc}
C_{0}^{-1} & 0\\
0 & \left(C_{2}-EC_{0}^{-1}E^{*}\right)^{-1}
\end{array}\right)
\end{align*}
and so
\begin{align*}
\left(\begin{array}{cc}
C_{0} & E^{*}\\
E & C_{2}
\end{array}\right)^{-1} & =\left(\begin{array}{cc}
1 & -C_{0}^{-1}E^{*}\\
0 & 1
\end{array}\right)\left(\begin{array}{cc}
C_{0}^{-1} & 0\\
0 & \left(C_{2}-EC_{0}^{-1}E^{*}\right)^{-1}
\end{array}\right)\left(\begin{array}{cc}
1 & 0\\
-EC_{0}^{-1} & 1
\end{array}\right)\\
 & =\left(\begin{array}{cc}
C_{0}^{-1} & -C_{0}^{-1}E^{*}\left(C_{2}-EC_{0}^{-1}E^{*}\right)^{-1}\\
-\left(C_{2}-EC_{0}^{-1}E^{*}\right)^{-1}EC_{0}^{-1} & \left(C_{2}-EC_{0}^{-1}E^{*}\right)^{-1}
\end{array}\right)
\end{align*}
With this and abbreviating 
\[
C_{1}\coloneqq C_{2}-EC_{0}^{-1}E^{*}
\]
the operator coefficient $M_{0}$ becomes
\[
M_{0}=\left(\begin{array}{cccc}
\varrho_{0} & 0 & 0 & 0\\
0 & \varrho_{1} & 0 & 0\\
0 & 0 & C_{0}^{-1}+C_{0}^{-1}E^{*}C_{1}^{-1}EC_{0}^{-1} & -C_{0}^{-1}E^{*}C_{1}^{-1}\\
0 & 0 & -C_{1}^{-1}EC_{0}^{-1} & C_{1}^{-1}
\end{array}\right)
\]
and for $M_{1}$ we get 
\[
\left(\begin{array}{cccc}
0 & \quad0 & \quad0 & \quad-\Lambda^{*}\eta_{1}^{*}\\
0 & \quad0 & \quad-\Lambda^{*}\eta_{0}^{*} & \quad0\\
0 & \quad\eta_{0}\Lambda & \quad0 & \quad0\\
\eta_{1}\Lambda & \quad0 & \quad0 & \quad0
\end{array}\right),
\]
where the parameters $\eta_{0},\eta_{1}$ have been inserted for flexibility.
In \cite{1136.74018} we find the case $\eta_{0}=\eta_{1}=1$. For
$\eta_{0}=1$ and $\eta_{1}=0$ we recover a variant of the micropolar
media case with
\[
\left(\begin{array}{c}
\sigma\\
\tau
\end{array}\right)=\left(\begin{array}{cc}
C_{0} & E^{*}\\
E & C_{2}
\end{array}\right)\left(\left(\begin{array}{c}
\nabla u\\
\nabla\phi
\end{array}\right)+\left(\begin{array}{c}
\phi\times\\
0
\end{array}\right)\right)
\]
as the modified Hooke's law. The case $E=0$ corresponds to the case
discussed in (\ref{eq:Cosserat}). In any case the resulting evolutionary
equation
\[
\left(\partial_{0}M_{0}+M_{1}+\left(\begin{array}{cccc}
0 & 0 & -\nabla\cdot & 0\\
0 & 0 & 0 & -\nabla\cdot\\
-\interior\nabla & 0 & 0 & 0\\
0 & -\interior\nabla & 0 & 0
\end{array}\right)\right)\left(\begin{array}{c}
\mathbf{v}\\
\mathbf{w}\\
\sigma\\
\tau
\end{array}\right)=\left(\begin{array}{c}
\mathbf{f}\\
\mathbf{g}\\
0\\
0
\end{array}\right)
\]
for hemitropic media fits into the scheme of our solution theory,
here with underlying Hilbert space
\[
H=L^{2,1}\left(\Omega\right)\oplus L^{2,1}\left(\Omega\right)\oplus L^{2,2}\left(\Omega\right)\oplus L^{2,2}\left(\Omega\right).
\]

\end{rem}

\subsection{Other Descendants of Micromorphic Media\label{sub:Other-Descendants-of}}

Having applied the ``mother'' and ``descendant''-mechanism in
more involved situations, we only roughly state the connections of
the following equations with the model of micromorphic media by stating
the operator $B$, which transforms the model for micromorphic media
to the one under consideration by being a $(B)$-descendant.

The system of elastic equations read as follows
\[
\partial_{0}^{2}u-\nabla\cdot\sigma=\mathbf{f}\in H_{\rho,0}(\mathbb{R};L^{2,1}(\Omega))
\]

together with Hook's law:
\[
\sigma=C\nabla u,
\]
\[
K=C^{-1}
\]
where $K=C^{-1}$ is the compliance. We abbreviate $\partial_{0}u\eqqcolon\mathbf{v}.$

The material properties ($C\mathrm{skew}=\mathrm{skew}C=0$) usually
assumed and encoded in $C$ suggests to follow the considerations
of Remark \ref{rem:degenerate-materials} and so to consider only
the symmetric part $T=\iota_{\mathrm{sym}}^{*}\sigma$ of $\sigma$
as our actual unknown. With $C_{0}\coloneqq\mathrm{\iota_{sym}^{*}}C\mathrm{\iota_{sym}}$
we get
\begin{align*}
T & =C_{0}\left(\mathrm{\iota_{sym}^{*}}\nabla\right)u.
\end{align*}
According to Remark \ref{rem:degenerate-materials}, compare \cite{Picard201354},
we are therefore rather led to consider the reduced system in $L^{2,1}\left(\Omega\right)\oplus\mathrm{sym}\left[L^{2,2}\left(\Omega\right)\right]:$
\[
\left(\partial_{0}\left(\begin{array}{cc}
\varrho_{0} & 0\\
0 & C_{0}^{-1}
\end{array}\right)+\left(\begin{array}{cc}
0 & -\Div\\
-\interior\Grad & 0
\end{array}\right)\right)\left(\begin{array}{c}
\mathbf{v}\\
T
\end{array}\right)=\left(\begin{array}{c}
\mathbf{f}\\
0
\end{array}\right).
\]
Here we have utilized the abbreviations $\Div\coloneqq\left(\nabla\cdot\right)\iota_{\mathrm{sym}}$,~
$\interior\Grad\coloneqq\overline{\iota_{\mathrm{sym}}^{*}\interior\nabla}$
in keeping the notation used in \cite{pre05919306}. Now, it is easily
verified that the system of classical elasticity discussed is a $(B)$-descendant
of the systems for micromorphic media with 
\[
B=\left(\begin{array}{ccccc}
1 & 0 & 0 & 0 & 0\\
0 & \iota_{0}^{*} & 0 & 0 & 0\\
0 & 0 & \iota_{\mathrm{sym}}^{*} & 0 & 0\\
0 & 0 & 0 & \iota_{0}^{*} & 0\\
0 & 0 & 0 & 0 & \iota_{0}^{*}
\end{array}\right)
\]
such that the underlying Hilbert space becomes $H=L^{2,1}\left(\Omega\right)\oplus\left\{ 0\right\} \oplus\mathrm{sym}\left[L^{2,2}\left(\Omega\right)\right]\oplus\left\{ 0\right\} \oplus\left\{ 0\right\} .$
Indeed, we only check, whether the operator containing the spatial
derivatives admits the form desired:
\begin{align*}
 & \overline{\left(\begin{array}{ccccc}
1 & 0 & 0 & 0 & 0\\
0 & \iota_{0}^{*} & 0 & 0 & 0\\
0 & 0 & \iota_{\mathrm{sym}}^{*} & 0 & 0\\
0 & 0 & 0 & \iota_{0}^{*} & 0\\
0 & 0 & 0 & 0 & \iota_{0}^{*}
\end{array}\right)\left(\begin{array}{ccccc}
0 & 0 & -\nabla\cdot & 0 & 0\\
0 & 0 & 0 & -\nabla\cdot & 0\\
-\interior\nabla & 0 & 0 & 0 & 0\\
0 & -\interior\nabla & 0 & 0 & 0\\
0 & 0 & 0 & 0 & 0
\end{array}\right)}\left(\begin{array}{ccccc}
1 & 0 & 0 & 0 & 0\\
0 & \iota_{0} & 0 & 0 & 0\\
0 & 0 & \iota_{\mathrm{sym}} & 0 & 0\\
0 & 0 & 0 & \iota_{0} & 0\\
0 & 0 & 0 & 0 & \iota_{0}
\end{array}\right)=\\
 & =\left(\begin{array}{ccccc}
0 & 0 & -\nabla\cdot\iota_{\mathrm{sym}} & 0 & 0\\
0 & 0 & 0 & 0 & 0\\
-\overline{\iota_{\mathrm{sym}}^{*}\interior\nabla} & 0 & 0 & 0 & 0\\
0 & 0 & 0 & 0 & 0\\
0 & 0 & 0 & 0 & 0
\end{array}\right).
\end{align*}

\begin{rem}
(a) It may be interesting to note that by Korn's inequality we know
that in this case the closure bar does not add anything to the domain
of $\interior\Grad$ in comparison with $\interior\nabla$ and so
actually%
\[
\overline{\mathrm{\iota_{sym}^{*}}\interior\nabla}=\mathrm{\iota_{sym}^{*}}\interior\nabla.
\]
(b) For sake of comparison let us consider the simple isotropic case:
$C_{0}=2\mu\mathrm{sym}+3\lambda\mathbb{P}=2\mu\mathrm{sym}_{0}+\left(3\lambda+2\mu\right)\mathbb{P}$
with $\lambda,\mu\in\mathbb{R}$. Positive definiteness of $C_{0}$
(and so of $C_{0}^{-1}$) is characterized by
\[
\mu>0,\:\lambda+\frac{2}{3}\mu>0.
\]

\end{rem}
In the same spirit as above, we can construct other models stemming
from the model of micromorphic media, by assuming additional constraints
on the material. As for instance, assume the symmetry of the stress.
Then a natural choice for $B$ is 
\[
\left(\begin{array}{ccccc}
1 & 0 & 0 & 0 & 0\\
0 & \iota_{\mathrm{skew}}^{*} & 0 & 0 & 0\\
0 & 0 & \iota_{\mathrm{sym}}^{*} & 0 & 0\\
0 & 0 & 0 & 1\otimes\iota_{\mathrm{skew}}^{*} & 0\\
0 & 0 & 0 & 0 & \iota_{0}^{*}
\end{array}\right).
\]
This choice results in the following leading term $M_{0}$ 
\[
\left(\begin{array}{ccccc}
\varrho_{0} & 0 & 0 & 0 & 0\\
0 & \iota_{\mathrm{skew}}^{*}\varrho_{2}\iota_{\mathrm{skew}} & 0 & 0 & 0\\
0 & 0 & \iota_{\mathrm{sym}}^{*}W_{00}\iota_{\mathrm{sym}} & \iota_{\mathrm{sym}}^{*}W_{01}\left(1\otimes\iota_{\mathrm{skew}}\right) & 0\\
0 & 0 & \left(1\otimes\iota_{\mathrm{skew}}^{*}\right)\left(W_{01}\right)^{*}\iota_{\mathrm{sym}} & \left(1\otimes\iota_{\mathrm{skew}}^{*}\right)W_{11}\left(1\otimes\iota_{\mathrm{skew}}\right) & 0\\
0 & 0 & 0 & 0 & 0
\end{array}\right)
\]
and the following operator containing the spatial derivatives
\begin{align*}
 & \overline{\left(\begin{array}{ccccc}
1 & 0 & 0 & 0 & 0\\
0 & \iota_{\mathrm{skew}}^{*} & 0 & 0 & 0\\
0 & 0 & \iota_{\mathrm{sym}}^{*} & 0 & 0\\
0 & 0 & 0 & 1\otimes\iota_{\mathrm{skew}}^{*} & 0\\
0 & 0 & 0 & 0 & \iota_{0}^{*}
\end{array}\right)\left(\begin{array}{ccccc}
0 & 0 & -\nabla\cdot & 0 & 0\\
0 & 0 & 0 & -\nabla\cdot & 0\\
-\interior\nabla & 0 & 0 & 0 & 0\\
0 & -\interior\nabla & 0 & 0 & 0\\
0 & 0 & 0 & 0 & 0
\end{array}\right)}\left(\begin{array}{ccccc}
1 & 0 & 0 & 0 & 0\\
0 & \iota_{\mathrm{skew}} & 0 & 0 & 0\\
0 & 0 & \iota_{\mathrm{sym}} & 0 & 0\\
0 & 0 & 0 & 1\otimes\iota_{\mathrm{skew}} & 0\\
0 & 0 & 0 & 0 & \iota_{0}
\end{array}\right)=\\
 & =\left(\begin{array}{ccccc}
0 & 0 & -\nabla\cdot\iota_{\mathrm{sym}} & 0 & 0\\
0 & 0 & 0 & -\overline{\iota_{\mathrm{skew}}^{*}\nabla\cdot}\left(1\otimes\iota_{\mathrm{skew}}\right) & 0\\
-\overline{\iota_{\mathrm{sym}}^{*}\interior\nabla} & 0 & 0 & 0 & 0\\
0 & -\overline{\left(1\otimes\iota_{\mathrm{skew}}^{*}\right)\interior\nabla}\iota_{\mathrm{skew}} & 0 & 0 & 0\\
0 & 0 & 0 & 0 & 0
\end{array}\right).
\end{align*}

\begin{rem}
If we assume that the coefficients are such that the resulting systems
in Section \ref{sec:Cosserat-Elasticity} and Section \ref{sub:Micromorphic-Media}
are well-posed via the assumption that $M_{0}$ is strictly positive
definite, we can conclude well-posedness of the descendant model by
the general mechanism given in Corollary \ref{cor:well-posedness_descendant}.
However, it may be that the original system may fail to be well-posed.
Only after projecting onto the right Hilbert spaces, the equations
become well-posed. In Remark \ref{rem:degenerate-materials} we have
discussed the general issue of degenerate $M_{0}$ and of how to obtain
a well-posed equation by assuming that $M_{0}$ is strictly positive
on its range and projecting via the descendant mechanism onto the
orthogonal complement of $M_{0}$.
\end{rem}
For deriving a model discussed in \cite{1104.74007}, we are led to
consider the operator
\[
B=\left(\begin{array}{ccccc}
1 & 0 & 0 & 0 & 0\\
0 & \iota_{\mathrm{sym}_{0}}^{*} & 0 & 0 & 0\\
0 & 0 & 1 & 0 & 0\\
0 & 0 & 0 & 1\otimes\iota_{\mathrm{sym}_{0}}^{*} & 0\\
0 & 0 & 0 & 0 & \iota_{0}^{*}
\end{array}\right)
\]
as the compatible operator giving the descendant mechanism of the
micromorphic media model. The underlying Hilbert space of the descendant
model is
\[
H=L^{2,1}\left(\Omega\right)\oplus\mathrm{sym}_{0}\left[L^{2,2}\left(\Omega\right)\right]\oplus L^{2,2}\left(\Omega\right)\oplus\left(1\otimes\mathrm{sym}_{0}\right)\left[L^{2,3}\left(\Omega\right)\right]\oplus\left\{ 0\right\} .
\]
In this case $M_{1}$ admits the form
\begin{align*}
 & \left(\begin{array}{ccccc}
1 & 0 & 0 & 0 & 0\\
0 & \iota_{\mathrm{sym}_{0}}^{*} & 0 & 0 & 0\\
0 & 0 & 1 & 0 & 0\\
0 & 0 & 0 & 1\otimes\iota_{\mathrm{sym}_{0}}^{*} & 0\\
0 & 0 & 0 & 0 & \iota_{0}^{*}
\end{array}\right)\left(\begin{array}{ccccc}
0 & 0 & 0 & 0 & 0\\
0 & 0 & -\mathrm{skew} & 0 & -\iota_{\mathrm{sym}}\\
0 & \mathrm{skew} & 0 & 0 & 0\\
0 & 0 & 0 & 0 & 0\\
0 & \iota_{\mathrm{sym}}^{*} & 0 & 0 & 0
\end{array}\right)\left(\begin{array}{ccccc}
1 & 0 & 0 & 0 & 0\\
0 & \iota_{\mathrm{sym}_{0}} & 0 & 0 & 0\\
0 & 0 & 1 & 0 & 0\\
0 & 0 & 0 & 1\otimes\iota_{\mathrm{sym}_{0}} & 0\\
0 & 0 & 0 & 0 & \iota_{0}
\end{array}\right)=\\
 & =\left(\begin{array}{ccccc}
0 & 0 & 0 & 0 & 0\\
0 & 0 & 0 & 0 & 0\\
0 & 0 & 0 & 0 & 0\\
0 & 0 & 0 & 0 & 0\\
0 & 0 & 0 & 0 & 0
\end{array}\right)
\end{align*}
The new $M_{0}$ is given by
\begin{align*}
 & \left(\begin{array}{ccccc}
1 & 0 & 0 & 0 & 0\\
0 & \iota_{\mathrm{sym}_{0}}^{*} & 0 & 0 & 0\\
0 & 0 & 1 & 0 & 0\\
0 & 0 & 0 & 1\otimes\iota_{\mathrm{sym}_{0}}^{*} & 0\\
0 & 0 & 0 & 0 & \iota_{0}^{*}
\end{array}\right)\left(\begin{array}{ccccc}
\varrho_{0} & 0 & 0 & 0 & 0\\
0 & \varrho_{2} & 0 & 0 & 0\\
0 & 0 & W_{00} & W_{01} & W_{02}\\
0 & 0 & W_{01}^{*} & W_{11} & W_{12}\\
0 & 0 & W_{02}^{*} & W_{12}^{*} & W_{22}
\end{array}\right)\left(\begin{array}{ccccc}
1 & 0 & 0 & 0 & 0\\
0 & \iota_{\mathrm{sym}_{0}}^{*} & 0 & 0 & 0\\
0 & 0 & 1 & 0 & 0\\
0 & 0 & 0 & 1\otimes\iota_{\mathrm{sym}_{0}}^{*} & 0\\
0 & 0 & 0 & 0 & \iota_{0}^{*}
\end{array}\right)=\\
 & =\left(\begin{array}{ccccc}
\varrho_{0} & 0 & 0 & 0 & 0\\
0 & \iota_{\mathrm{sym_{0}}}^{*}\varrho_{2}\iota_{\mathrm{sym_{0}}} & 0 & 0 & 0\\
0 & 0 & W_{00} & W_{01}\left(1\otimes\iota_{\mathrm{sym}_{0}}\right) & 0\\
0 & 0 & \left(1\otimes\iota_{\mathrm{sym}_{0}}^{*}\right)W_{01}^{*} & \left(1\otimes\iota_{\mathrm{sym}_{0}}^{*}\right)W_{11}\left(1\otimes\iota_{\mathrm{sym}_{0}}\right) & 0\\
0 & 0 & 0 & 0 & 0
\end{array}\right),
\end{align*}
whereas the operator containing the spatial derivatives reads as
\begin{align*}
 & \overline{\left(\begin{array}{ccccc}
1 & 0 & 0 & 0 & 0\\
0 & \iota_{\mathrm{sym}_{0}}^{*} & 0 & 0 & 0\\
0 & 0 & 1 & 0 & 0\\
0 & 0 & 0 & 1\otimes\iota_{\mathrm{sym}_{0}}^{*} & 0\\
0 & 0 & 0 & 0 & \iota_{0}^{*}
\end{array}\right)\left(\begin{array}{ccccc}
0 & 0 & -\nabla\cdot & 0 & 0\\
0 & 0 & 0 & -\nabla\cdot & 0\\
-\interior\nabla & 0 & 0 & 0 & 0\\
0 & -\interior\nabla & 0 & 0 & 0\\
0 & 0 & 0 & 0 & 0
\end{array}\right)}\left(\begin{array}{ccccc}
1 & 0 & 0 & 0 & 0\\
0 & \iota_{\mathrm{sym}_{0}}^{*} & 0 & 0 & 0\\
0 & 0 & 1 & 0 & 0\\
0 & 0 & 0 & 1\otimes\iota_{\mathrm{sym}_{0}}^{*} & 0\\
0 & 0 & 0 & 0 & \iota_{0}^{*}
\end{array}\right)=\\
 & =\left(\begin{array}{ccccc}
0 & 0 & -\nabla\cdot & 0 & 0\\
0 & 0 & 0 & -\overline{\iota_{\mathrm{sym}_{0}}^{*}\nabla\cdot}\left(1\otimes\iota_{\mathrm{sym}_{0}}\right) & 0\\
-\interior\nabla & 0 & 0 & 0 & 0\\
0 & -\overline{\left(1\otimes\iota_{\mathrm{sym}_{0}}^{*}\right)\interior\nabla}\iota_{\mathrm{sym}_{0}} & 0 & 0 & 0\\
0 & 0 & 0 & 0 & 0
\end{array}\right).
\end{align*}

A further descendant can be obtained, if, in the previous model, we
additionally assume the symmetry of stress. The operator $B$ in this
case reads as 
\[
B=\left(\begin{array}{ccccc}
1 & 0 & 0 & 0 & 0\\
0 & \iota_{\mathrm{sym}_{0}}^{*} & 0 & 0 & 0\\
0 & 0 & \iota_{\mathrm{sym}}^{*} & 0 & 0\\
0 & 0 & 0 & 1\otimes\iota_{\mathrm{sym}_{0}}^{*} & 0\\
0 & 0 & 0 & 0 & \iota_{0}^{*}
\end{array}\right)
\]
mapping onto the Hilbert space 
\[
H=L^{2,1}\left(\Omega\right)\oplus\mathrm{sym}_{0}\left[L^{2,2}\left(\Omega\right)\right]\oplus\mathrm{sym}\left[L^{2,2}\left(\Omega\right)\right]\oplus\left(1\otimes\mathrm{sym}_{0}\right)\left[L^{2,3}\left(\Omega\right)\right]\oplus\left\{ 0\right\} .
\]

The new $M_{0}$ is given by 
\[
\left(\begin{array}{ccccc}
\varrho_{0} & 0 & 0 & 0 & 0\\
0 & \iota_{\mathrm{sym_{0}}}^{*}\varrho_{2}\iota_{\mathrm{sym_{0}}} & 0 & 0 & 0\\
0 & 0 & \iota_{\mathrm{sym}}^{*}W_{00}\iota_{\mathrm{sym}} & \iota_{\mathrm{sym}}^{*}W_{01}\left(1\otimes\iota_{\mathrm{sym}_{0}}\right) & 0\\
0 & 0 & \left(1\otimes\iota_{\mathrm{sym}_{0}}^{*}\right)W_{01}^{*}\iota_{\mathrm{sym}} & \left(1\otimes\iota_{\mathrm{sym}_{0}}^{*}\right)W_{11}\left(1\otimes\iota_{\mathrm{sym}_{0}}\right) & 0\\
0 & 0 & 0 & 0 & 0
\end{array}\right)
\]
 and the resulting operator containing the spatial derivatives is
written as
\begin{align*}
 & \overline{\left(\begin{array}{ccccc}
1 & 0 & 0 & 0 & 0\\
0 & \iota_{\mathrm{sym}_{0}}^{*} & 0 & 0 & 0\\
0 & 0 & \iota_{\mathrm{sym}}^{*} & 0 & 0\\
0 & 0 & 0 & 1\otimes\iota_{\mathrm{sym}_{0}}^{*} & 0\\
0 & 0 & 0 & 0 & \iota_{0}^{*}
\end{array}\right)\left(\begin{array}{ccccc}
0 & 0 & -\nabla\cdot & 0 & 0\\
0 & 0 & 0 & -\nabla\cdot & 0\\
-\interior\nabla & 0 & 0 & 0 & 0\\
0 & -\interior\nabla & 0 & 0 & 0\\
0 & 0 & 0 & 0 & 0
\end{array}\right)}\left(\begin{array}{ccccc}
1 & 0 & 0 & 0 & 0\\
0 & \iota_{\mathrm{sym}_{0}}^{*} & 0 & 0 & 0\\
0 & 0 & \iota_{\mathrm{sym}} & 0 & 0\\
0 & 0 & 0 & 1\otimes\iota_{\mathrm{sym}_{0}}^{*} & 0\\
0 & 0 & 0 & 0 & \iota_{0}^{*}
\end{array}\right)=\\
 & =\left(\begin{array}{ccccc}
0 & 0 & -\nabla\cdot\iota_{\mathrm{sym}} & 0 & 0\\
0 & 0 & 0 & -\overline{\iota_{\mathrm{sym}_{0}}^{*}\nabla\cdot}\left(1\otimes\iota_{\mathrm{sym}_{0}}\right) & 0\\
-\overline{\iota_{\mathrm{sym}}^{*}\interior\nabla} & 0 & 0 & 0 & 0\\
0 & -\overline{\left(1\otimes\iota_{\mathrm{sym}_{0}}^{*}\right)\interior\nabla}\iota_{\mathrm{sym}_{0}} & 0 & 0 & 0\\
0 & 0 & 0 & 0 & 0
\end{array}\right).
\end{align*}
Note that the resulting model is still a coupled model.

\section{\label{sec:Microstretch-Models}Microstretch Models}

We conclude the article by stating a model differing from the models
already discussed in view of the mother and descendant mechanism.
The equations, taken from \cite{0718.73014,0164.27507}, read as follows

\begin{align}
 & \partial_{0}\left(\begin{array}{ccc}
\rho_{0} & 0 & 0\\
0 & \rho_{1} & 0\\
0 & 0 & \rho_{2}
\end{array}\right)\left(\begin{array}{c}
\dot{u}\\
\dot{\psi}\\
\dot{\varphi}
\end{array}\right)+\left(\begin{array}{ccc}
0 & 0 & 0\\
-\Lambda^{\ast} & 0 & 0\\
0 & 0 & 0
\end{array}\right)\left(\begin{array}{c}
\tau\\
\mu\\
\pi
\end{array}\right)+\left(\begin{array}{c}
0\\
0\\
\sigma
\end{array}\right)-\left(\begin{array}{ccc}
\nabla\cdot & 0 & 0\\
0 & \nabla\cdot & 0\\
0 & 0 & \nabla\cdot
\end{array}\right)\left(\begin{array}{c}
\tau\\
\mu\\
\pi
\end{array}\right)=\left(\begin{array}{c}
\mathbf{f}\\
\mathbf{g}\\
\mathbf{h}
\end{array}\right)\label{eq:stretch}\\
 & \in H_{\rho,0}(\mathbb{R};L^{2,1}(\Omega)\oplus L^{2,2}(\Omega)\oplus L^{2,2}(\Omega))\nonumber 
\end{align}
where $\Lambda\coloneqq2\iota_{\wedge}\ast$ (see Subsection \ref{sub:Micropolar-or-Cosserat}),
$\tau,\sigma\in H_{\rho,0}(\mathbb{R};L^{2,2}(\Omega)),\mu,\pi\in H_{\rho,0}(\mathbb{R};L^{2,3}(\Omega))$
and the $\rho_{i}$'s, $i\in\{0,1,2\}$, are bounded selfadjoint operators
in the respective spaces $L^{2,1}(\Omega)$ or $L^{2,2}(\Omega)$.
We find a constitutive relation of the form 
\begin{equation}
\left(\begin{array}{c}
\tau\\
\mu\\
\sigma\\
\pi
\end{array}\right)=\left(\begin{array}{cccc}
C_{0} & B & D & F\\
B^{*} & C_{1} & E & G\\
D^{*} & E^{*} & C_{2} & K\\
F^{*} & G^{*} & K^{*} & C_{3}
\end{array}\right)\left(\begin{array}{c}
e\\
\kappa\\
\varphi\\
\zeta
\end{array}\right),\label{eq:cont_rel_microstretch}
\end{equation}
where 
\begin{align}
e & =\nabla u-\Lambda\psi=\nabla u+\psi\times,\nonumber \\
\kappa & =\nabla\psi,\nonumber \\
\zeta & =\nabla\varphi.\label{eq:def_microstretch}
\end{align}
and suitable bounded linear operators $C_{0},B,D,F,C_{1},E,G,C_{2},K,C_{3}$.
In the following, we want to show that the latter model fits into
the general scheme of Theorem \ref{linear-evolutionary-solution-theory}.
For this we assume continuous invertibility of $\left(\begin{array}{ccc}
C_{0} & B & F\\
B^{*} & C_{1} & G\\
F^{*} & G^{*} & C_{3}
\end{array}\right)$ and denote its inverse by $W$. We have
\[
\left(\begin{array}{c}
\tau-D\varphi\\
\mu-E\varphi\\
\pi-K^{*}\varphi
\end{array}\right)=\left(\begin{array}{ccc}
C_{0} & B & F\\
B^{*} & C_{1} & G\\
F^{*} & G^{*} & C_{3}
\end{array}\right)\left(\begin{array}{c}
e\\
\kappa\\
\zeta
\end{array}\right)
\]
 and thus

\begin{align}
\left(\begin{array}{c}
e\\
\kappa\\
\zeta
\end{array}\right) & =W\left(\begin{array}{c}
\tau-D\varphi\\
\mu-E\varphi\\
\pi-K^{*}\varphi
\end{array}\right)\nonumber \\
 & =W\left(\begin{array}{c}
\tau\\
\mu\\
\pi
\end{array}\right)-W\left(\begin{array}{c}
D\\
E\\
K^{*}
\end{array}\right)\varphi.\label{eq:reduced_const_rel_stretch}
\end{align}
The latter equations together with the dynamic equations from above
result in the canonical form 
\begin{equation}
\left(\partial_{0}M_{0}+M_{1}+\partial_{0}^{-1}M_{2}+A\right)\left(\begin{array}{c}
\dot{u}\\
\dot{\psi}\\
\dot{\varphi}\\
\tau\\
\mu\\
\pi
\end{array}\right)=\left(\begin{array}{c}
\mathbf{f}\\
\mathbf{g}\\
\mathbf{h}\\
0\\
0\\
0
\end{array}\right),\label{eq:microstretch_evo}
\end{equation}
where
\begin{align*}
M_{0} & =\left(\begin{array}{cccc}
\rho_{0} & 0 & 0 & \begin{array}{ccc}
0 & 0 & 0\end{array}\\
0 & \rho_{1} & 0 & \begin{array}{ccc}
0 & 0 & 0\end{array}\\
0 & 0 & \rho_{2} & \begin{array}{ccc}
0 & 0 & 0\end{array}\\
\begin{array}{c}
0\\
0\\
0
\end{array} & \begin{array}{c}
0\\
0\\
0
\end{array} & \begin{array}{c}
0\\
0\\
0
\end{array} & W
\end{array}\right),\\
M_{1} & =\left(\begin{array}{cccc}
0 & 0 & 0 & \begin{array}{ccc}
\quad0\; & 0 & 0\end{array}\\
0 & 0 & 0 & \begin{array}{ccc}
-\Lambda^{\ast} & \,0 & 0\end{array}\\
0 & 0 & 0 & \left(\begin{array}{ccc}
D^{\ast} & E^{\ast} & K\end{array}\right)W\\
\begin{array}{c}
0\\
0\\
0
\end{array} & \begin{array}{c}
\Lambda\\
0\\
0
\end{array} & -W\left(\begin{array}{c}
D\\
E\\
K^{\ast}
\end{array}\right) & \begin{array}{ccc}
\quad0\; & 0 & 0\\
\quad0\; & 0 & 0\\
\quad0\; & 0 & 0
\end{array}
\end{array}\right),\\
M_{2} & =\left(\begin{array}{cccc}
0 & 0 & 0 & \begin{array}{ccc}
0 & 0 & 0\end{array}\\
0 & 0 & 0 & \begin{array}{ccc}
0 & 0 & 0\end{array}\\
0 & 0 & C_{2}-\left(\begin{array}{ccc}
D^{\ast} & E^{\ast} & K\end{array}\right)W\left(\begin{array}{c}
D\\
E\\
K^{\ast}
\end{array}\right) & \begin{array}{ccc}
0 & 0 & 0\end{array}\\
\begin{array}{c}
0\\
0\\
0
\end{array} & \begin{array}{c}
0\\
0\\
0
\end{array} & \begin{array}{c}
0\\
0\\
0
\end{array} & \begin{array}{ccc}
0 & 0 & 0\\
0 & 0 & 0\\
0 & 0 & 0
\end{array}
\end{array}\right),\\
A & =\left(\begin{array}{cccccc}
0 & 0 & 0 & -\nabla\cdot & 0 & 0\\
0 & 0 & 0 & 0 & -\nabla\cdot & 0\\
0 & 0 & 0 & 0 & 0 & -\nabla\cdot\\
-\interior\nabla & 0 & 0 & 0 & 0 & 0\\
0 & -\interior\nabla & 0 & 0 & 0 & 0\\
0 & 0 & -\interior\nabla & 0 & 0 & 0
\end{array}\right).
\end{align*}
Indeed, the last three rows in (\ref{eq:microstretch_evo}) are (\ref{eq:reduced_const_rel_stretch})
combined with (\ref{eq:def_microstretch}) while the first two rows
of (\ref{eq:microstretch_evo}) are the first two rows of (\ref{eq:stretch}).
From (\ref{eq:cont_rel_microstretch}) and (\ref{eq:reduced_const_rel_stretch})
we read off that 
\begin{align*}
\sigma & =C_{2}\varphi+\left(\begin{array}{ccc}
D^{\ast} & E^{\ast} & K\end{array}\right)\left(\begin{array}{c}
e\\
\kappa\\
\zeta
\end{array}\right)\\
 & =C_{2}\varphi+\left(\begin{array}{ccc}
D^{\ast} & E^{\ast} & K\end{array}\right)W\left(\begin{array}{c}
\tau\\
\mu\\
\pi
\end{array}\right)-\left(\begin{array}{ccc}
D^{\ast} & E^{\ast} & K\end{array}\right)W\left(\begin{array}{c}
D\\
E\\
K^{*}
\end{array}\right)\varphi\\
 & =\left(\begin{array}{ccc}
D^{\ast} & E^{\ast} & K\end{array}\right)W\left(\begin{array}{c}
\tau\\
\mu\\
\pi
\end{array}\right)+\partial_{0}^{-1}\left(C_{2}-\left(\begin{array}{ccc}
D^{\ast} & E^{\ast} & K\end{array}\right)W\left(\begin{array}{c}
D\\
E\\
K^{*}
\end{array}\right)\right)\dot{\varphi},
\end{align*}
and thus, the third row of (\ref{eq:microstretch_evo}) gives the
third row in (\ref{eq:stretch}). 
\begin{thm}
Assume that $\rho_{i},\, i\in\{0,1,2\},$ and $W$ are strictly positive
definite. Then there exists $\rho_{0}\in\mathbb{R}_{>0}$ such that
for every $\rho\geq\rho_{0}$ and $(\mathbf{f},\mathbf{g},\mathbf{h})\in H_{\rho,0}(\mathbb{R};L^{2,1}(\Omega)\oplus L^{2,2}(\Omega)\oplus L^{2,2}(\Omega))$
there exists uniquely determined $(\dot{u},\dot{\psi},\dot{\varphi},\tau,\mu,\pi)\in H_{\rho,0}(\mathbb{R};L^{2,1}(\Omega)\oplus L^{2,2}(\Omega)\oplus L^{2,2}(\Omega)\oplus L^{2,2}(\Omega)\oplus L^{2,3}(\Omega)\oplus L^{2,3}(\Omega))$
satisfying (\ref{eq:microstretch_evo}). Moreover, the energy balance
\begin{multline*}
\left\langle \left(\begin{array}{c}
\dot{u}\\
\dot{\psi}\\
\dot{\varphi}\\
\tau\\
\mu\\
\pi
\end{array}\right)\left|\left(\begin{array}{c}
\rho_{0}\dot{u}\\
\rho_{1}\dot{\psi}\\
\rho_{2}\dot{\varphi}\\
W\left(\begin{array}{c}
\tau\\
\mu\\
\pi
\end{array}\right)
\end{array}\right)\right.\right\rangle (b)+\left\langle \varphi\left|\left(C_{2}-\left(\begin{array}{ccc}
D^{\ast} & E^{\ast} & K\end{array}\right)W\left(\begin{array}{c}
D\\
E\\
K^{\ast}
\end{array}\right)\right)\varphi\right.\right\rangle (b)\\
=\left\langle \left(\begin{array}{c}
\dot{u}\\
\dot{\psi}\\
\dot{\varphi}\\
\tau\\
\mu\\
\pi
\end{array}\right)\left|\left(\begin{array}{c}
\rho_{0}\dot{u}\\
\rho_{1}\dot{\psi}\\
\rho_{2}\dot{\varphi}\\
W\left(\begin{array}{c}
\tau\\
\mu\\
\pi
\end{array}\right)
\end{array}\right)\right.\right\rangle (a)+\left\langle \varphi\left|\left(C_{2}-\left(\begin{array}{ccc}
D^{\ast} & E^{\ast} & K\end{array}\right)W\left(\begin{array}{c}
D\\
E\\
K^{\ast}
\end{array}\right)\right)\varphi\right.\right\rangle (a)\\
+2\Re\intop_{a}^{b}\left\langle \left.\left(\begin{array}{c}
\dot{u}\\
\dot{\psi}\\
\dot{\varphi}
\end{array}\right)\right|\left(\begin{array}{c}
\mathbf{f}\\
\mathbf{g}\\
\mathbf{h}
\end{array}\right)\right\rangle 
\end{multline*}
holds for almost every $a,b\in\mathbb{R}.$\end{thm}
\begin{proof}
This is a direct consequence of Theorem \ref{linear-evolutionary-solution-theory}
and Theorem \ref{thm:energy balance}.\end{proof}
\begin{rem}
We note here that only by assuming 
\[
C_{2}=\left(\begin{array}{ccc}
D^{\ast} & E^{\ast} & K\end{array}\right)W\left(\begin{array}{c}
D\\
E\\
K^{\ast}
\end{array}\right)
\]
we get that 
\[
M_{2}=0
\]
which in turn results in the fact that the latter system of equations
is a $(B)$-mother of micromorphic media (see Subsection \ref{sub:Micromorphic-Media})
for 
\[
B=\left(\begin{array}{cccccc}
1 & 0 & 0 & 0 & 0 & 0\\
0 & 1 & 0 & 0 & 0 & 0\\
0 & 0 & \iota_{0}^{*} & 0 & 0 & 0\\
0 & 0 & 0 & 1 & 0 & 0\\
0 & 0 & 0 & 0 & 1 & 0\\
0 & 0 & 0 & 0 & 0 & \iota_{0}^{*}
\end{array}\right).
\]

\end{rem}

\section{Conclusion\label{sec:Conclusion}}

We discussed several models from elasticity and have shown that they
fit into the same solution scheme given by Theorem \ref{linear-evolutionary-solution-theory}.
Moreover, we showed that energy is conserved with the help of Theorem
\ref{thm:energy balance}. Furthermore, we provided a mathematically
rigorous concept of deriving models from a given one. We proved that
many models of elasticity may be derived from the model for micromorphic
media, which under additional constraints on the material parameters
is a descendant of a model for microstretch.


\begin{thebibliography}{10}
\bibitem{zbMATH06233830} Holm {Altenbach}; Gérard~A. {Maugin}
and Vladimir {Erofeev}, editors. \newblock {\em {Mechanics of
generalized continua.}} \newblock Berlin: Springer, 2011.

\bibitem{40.0862.02} E.~Cosserat and F.~Cosserat. \newblock {Th{é}orie
des corps d{é}formables.} \newblock 1909.

\bibitem{2013} J.~Epperlein. \newblock Personal communication.
\newblock 2013.

\bibitem{0164.27507} A.C.~Eringen. \newblock {Micropolar fluid
with stretch.} \newblock {\em Int. J. Eng. Sci.}, 7:115--127,
1969.

\bibitem{zbMATH03352829} A. C. {Eringen}. \newblock {Mechanics
of continua.} \newblock {New York-London-Sydney: John Wiley \&
Sons, Inc. XVIII, 502 p.}, 1967.

\bibitem{0718.73014} A.~Eringen. \newblock {Theory of thermo-microstretch
elastic solids.} \newblock {\em Int. J. Eng. Sci.}, 28(12):1291--1301,
1990.

\bibitem{zbMATH01817639} A.Cemal {Eringen}. \newblock {\em {Nonlocal
continuum field theories.}} \newblock New York, NY: Springer, 2002.

\bibitem{Kalauch2011} A.~Kalauch, R.~Picard, S.~Siegmund, S.~Trostorff,
and M.~Waurick. \newblock {A Hilbert Space Perspective on Ordinary
Differential Equations with Memory Term}. \newblock {\em Journal
of Dynamics and Differential Equations.}, 2013. \newblock Accepted,
\url{http://arxiv.org/abs/1204.2924}.

\bibitem{MR0162404} R.~V. Kuvshinskii and {É}.~L. A{é}ro. \newblock
{Continuum theory of asymmetric elasticity. The problem of internal
rotation}. \newblock {\em Soviet Physics Solid State}, 5:1892--1897,
1963.

\bibitem{Kwapien} S.~Kwapien. \newblock{Isomorphic characterizations
of inner product spaces by orthogonal series with vector valued coefficients.} \newblock {\em
Stud. Math.}, 44:583--595, 1972.

\bibitem{0119.40302} R.~Mindlin. \newblock {Micro-structure in
linear elasticity.} \newblock {\em Arch. Ration. Mech. Anal.},
16:51--78, 1964.

\bibitem{1136.74018} D.~Natroshvili, R.~Gachechiladze, A.~Gachechiladze,
and I.~G. Stratis. \newblock {Transmission problems in the theory
of elastic hemitropic materials.} \newblock {\em Appl. Anal.},
86(12):1463--1508, 2007.

\bibitem{1104.74007} P.~Neff. \newblock {The Cosserat couple modulus
for continuous solids is zero viz the linearized Cauchy-stress tensor
is symmetric.} \newblock {\em Z. Angew. Math. Mech.}, 86(11):892--912,
2006.

\bibitem{zbMATH03456512} W.~{Nowacki}. \newblock {Theory of
micropolar elasticity. Course held at the Department for Mechanics
of Deformable Bodies, July 1970, Udine.} \newblock {International
Centre for Mechanical Sciences. Courses and Lectures. No.25. Wien
- New York: Springer-Verlag. 286 p. }, 1972.

\bibitem{zbMATH03490546}W.~{Nowacki}. \newblock {Micropolar
elasticity. Symposium organized by the Department of Mechanics of
Solids, June 1972. Udine 1974.} \newblock {International Centre
for Mechanical Sciences. Courses and Lectures. No.151. Wien - New
York: Springer-Verlag. IV, 168 p.}, 1974.

\bibitem{0604.73020} W.~Nowacki. \newblock {\em {Theory of asymmetric
elasticity. Transl. from the Polish by H. Zorski.}} \newblock {Oxford
etc.: Pergamon Press; Warszawa: PWN-Polish Scientific Publishers.
VIII, 383 p.}, 1986.

\bibitem{1200.35050} R.~Picard. \newblock {A structural observation
for linear material laws in classical mathematical physics.} \newblock
{\em Math. Methods Appl. Sci.}, 32(14):1768--1803, 2009.

\bibitem{Picard2012} R.~Picard. \newblock {A class of evolutionary
problems with an application to acoustic waves with impedance type
boundary conditions.} \newblock In W.~A. et~al, editor, {\em
{Spectral theory, mathematical system theory, evolution equations,
differential and difference equations. Selected papers of 21st international
workshop on operator theory and applications, IWOTA10, Berlin, Germany,
July 12--16, 2010.}}, volume 221 of {\em Operator Theory: Advances
and Applications}, pages 533--548, Basel, 2012. Birkh{ä}user.

\bibitem{Picard201354} R.~Picard. \newblock {Mother Operators
and their Descendants}. \newblock {\em Journal of Mathematical
Analysis and Applications}, 403(1):54--62, 2013.

\bibitem{pre05919306} R.~Picard and D.~McGhee. \newblock {\em
{Partial differential equations. A unified Hilbert space approach.}}
\newblock {de Gruyter Expositions in Mathematics 55. Berlin: de
Gruyter. xviii, 469~p.}, 2011.

\bibitem{PicTroWau2012c} R.~Picard, S.~Trostorff, and M.~Waurick.
\newblock A functional analytic perspective to delay differential
equations. \newblock {\em Operators and Matrices}, 2012. \newblock
{Special issue for the conference Spectral Theory and Differential
Operators, accepted. \url{http://arxiv.org/abs/1211.3894}}.

\bibitem{PicTroWau2012d} R.~Picard, S.~Trostorff, and M.~Waurick.
\newblock On a comprehensive class of linear control problems. \newblock
Technical report, TU Dresden, 2012. \newblock Submitted, \url{http://arxiv.org/abs/1208.3140}.

\bibitem{PicTroWau2012} R.~Picard, S.~Trostorff, and M.~Waurick.
\newblock {A note on a class of conservative, well-posed linear
control systems.} \newblock In M.~Reissig and M.~Ruzhansky, editors,
{\em Progress in Partial Differential Equations: Asymptotic Profiles,
Regularity and Well-Posedness}, volume~44 of {\em Springer Proceedings
in Mathematics and Statistics}, pages 261--286, Heidelberg, 2013.
Springer.

\bibitem{PTW_frac} R.~Picard, S.~Trostorff, and M.~Waurick. \newblock
On evolutionary equations with material laws containing fractional
integrals. \newblock Technical Report MATH-AN-05-2013, TU Dresden,
2013. \newblock Submitted, \url{http://arxiv.org/abs/1304.7620}.

\bibitem{RainerPicard2013} R.~Picard, S.~Trostorff, M.~Waurick,
and M.~Wehowski. \newblock On non-autonomous evolutionary problems.
\newblock {\em J. Evol. Equ. 13}, no. 4, 751-776. 2013. 

\bibitem{Rudin} W.~Rudin. \newblock {\em{ Real and complex analysis
}} \newblock {Mathematical Series. McGraw-Hill}, 1987. 

\bibitem{Trostorff2013} S.~Trostorff. \newblock Autonomous evolutionary
inclusions with applications to problems with nonlinear boundary conditions.
\newblock {\em Int. J. Pure Appl. Math.}, 85(2):303--338, 2013.

\bibitem{S.Trostorff} S.~Trostorff and M.~Wehowski. \newblock
Well-posedness of non-autonomous evolutionary inclusions. \newblock
Technical report, TU Dresden, 2013. \newblock Submitted, \url{http://arxiv.org/abs/1307.2074}.

\bibitem{Waurick2013c} M.~Waurick. \newblock On non-autonomous
integro-differential-algebraic evolutionary problems. \newblock Technical
report, TU Dresden, 2013. \newblock {Submitted, \url{http://arxiv.org/abs/1307.2429}}.\end{thebibliography}
\end{document}